\documentclass[review,english]{elsarticle_mod}
\usepackage{graphicx}
\usepackage{geometry}
\PassOptionsToPackage{normalem}{ulem}
\usepackage{ulem}
\usepackage{babel}
\usepackage{color}
\usepackage{amsmath}
\usepackage{amssymb}
\usepackage{amsthm}
\usepackage{hyperref}

\makeatletter
\numberwithin{equation}{section}
\numberwithin{figure}{section}
  \theoremstyle{remark}
  \newtheorem*{notation*}{\protect\notationname}
\theoremstyle{plain}
\newtheorem{thm}{\protect\theoremname}
  \theoremstyle{definition}
  
  \theoremstyle{plain}
  \newtheorem{lem}[thm]{\protect\lemmaname}
  \theoremstyle{remark}
  \newtheorem{rem}[thm]{\protect\remarkname}

\@ifundefined{showcaptionsetup}{}{%
 \PassOptionsToPackage{caption=false}{subfig}}
\usepackage{subfig}
\makeatother

  \providecommand{\definitionname}{Definition}
  \providecommand{\lemmaname}{Lemma}
  \providecommand{\notationname}{Notation}
  \providecommand{\remarkname}{Remark}
\providecommand{\theoremname}{Theorem}

\begin{document}

\begin{frontmatter}

\title{Reconstruction of the electric field of the Helmholtz equation in 3D}


\author[aff1]{Huy Tuan Nguyen\corref{mycorrespondingauthor}}
\cortext[mycorrespondingauthor]{Corresponding author}
\ead{nguyenhuytuan@tdt.edu.vn}

\author[aff2]{Vo Anh Khoa}
\ead{khoa.vo@gssi.infn.it, vakhoa.hcmus@gmail.com}

\author[aff3]{Mach Nguyet Minh}
\ead{mach@math.uni-frankfurt.de}

\author[aff4]{Thanh Tran}
\ead{thanh.tran@unsw.edu.au}

\address[aff1]{Applied Analysis Research Group, Faculty of Mathematics and Statistics, Ton Duc Thang University, Ho Chi Minh City, Vietnam}
\address[aff2]{Mathematics and Computer Science Division,
Gran Sasso Science Institute, L'Aquila, Italy}
\address[aff3]{Department of Mathematics, Goethe University Frankfurt, Germany}
\address[aff4]{School of Mathematics and Statistics, The University of New South Wales, Sydney, Australia}

\begin{abstract}
In this paper, we rigorously investigate the truncation method for the Cauchy problem
of Helmholtz equations which is widely used to model propagation phenomena in physical applications. The method is a well-known approach to the regularization of  several types of ill-posed problems, including the model postulated by Regi\' nska and Regi\' nski \cite{RR06}. Under certain specific assumptions, we examine the ill-posedness of the non-homogeneous problem by exploring the representation of solutions based on Fourier mode. Then the so-called regularized solution is established with respect to a frequency bounded by an appropriate regularization parameter. Furthermore, we provide
a short analysis of the nonlinear forcing term. The main results show the stability as well as the strong convergence confirmed by the error estimates in $L^2$-norm of such regularized solutions. Besides, the regularization parameters are formulated properly. Finally,
some illustrative examples are provided to corroborate our qualitative analysis.
\end{abstract}


\end{frontmatter}


\section{Introduction}\label{Sec:intro}

The scalar Helmholtz equation is the well-spring of many streams in both mathematical and engineering problems due to the formal equivalence of the wave equation (and the Schr\" odinger equation for further applications). To set the stage for our problem presented in this paper,  we review very briefly the relation between these equations by direct and  fundamental techniques. In fact, the scalar wave equation that derives the Helmholtz equation is simply expressed by
\begin{equation}
v_{tt}=c^{2}\Delta v+F\left(t,x\right),\label{eq:waveeqn}
\end{equation}
where $c$ denotes the local speed of propagation for waves, and $F\left(t,x\right)$
is a source that injects waves into the solution. Suppose that  we look
for a solution with the wave number $k=1/\lambda>0$ defined by wavelength
$\lambda$, and that the source generates waves of this type, i.e.
\[
v\left(t,x\right)=u\left(x\right)e^{-ikt},\quad F\left(t,x\right)=q\left(x\right)e^{-ikt}.
\]
Substituting these quantities into (\ref{eq:waveeqn}), then dividing by $e^{-ikt}$
and reordering the terms, we obtain
\[
\Delta u\left(x\right)+\frac{k^{2}}{c^{2}}u\left(x\right)=-\frac{q\left(x\right)}{c^{2}}.
\]
This is the three-dimensional non-homogeneous Helmholtz equation in
which we are interested here, which mathematically reads
\begin{eqnarray}
\Delta u\left(x\right)+k^{2}nu\left(x\right) =  -f\left(x\right),\label{eq:1.2}
\end{eqnarray}
where the coefficient $n=1/c^{2}$ is normalized in this paper and in principle known as the
index of refraction (\cite{HMKWH05,Kir11});
and $f\left(x\right)=q\left(x\right)/c^{2}$ represents the forcing
term.

In this paper, we continue the work that commenced in \cite{RR06} by Regi\' nska and Regi\' nski, where the Cauchy problem of Helmholtz equations in particular drives us to the model of reconstruction of the whole radiation field in optoelectronics. This problem is associated with Hadamard-instability due to the fact that the high frequency modes grow exponentially fast. Typically, it would imply the severe ill-posedness as well as the impossibility of solving the problem. Hence, it is customary to overcome this difficulty via a regularization method.

In the context of regularization methods, the homogeneous problem ($f\equiv0$) has been studied mathematically for more than a decade. Recent developments of theoretical computations have been achieved, e.g. the truncation method in \cite{RR06}, the quasi-reversibility-type method in \cite{Xiong10}, and the Tikhonov-type method in \cite{FFC11,QWS09}, where the energy of solution is supposedly known in some certain cases. On the other side, various boundary element regularization methods are solidly compared in \cite{Mar04}. We, nevertheless, stress that the qualitative analysis of instability from the above-mentioned works mostly lacks theoretical validation. Even though the authors in \cite{RR06} rigorously investigated the discernible impact of physical parameters upon independence of solution on given data, a convincing example seems to be needed.

Currently, there have been many other fields of study where the Helmholtz-type equations can be greatly used, such as the influence of the frequency on the stability of Cauchy problems \cite{IK11}, finding the shape of a part of a boundary in \cite{CF09}, regularization of the modified Helmholtz equation in \cite{NTN13} and the problem of identifying source functions in \cite{BN11,LL13}. As we can see from the references, while the literature on the homogeneous problem is very extensive, there would have to emerge some potential field to consider the non-homogeneous problem (and the nonlinear case which we shall figure out later on). Even though the homogeneous problem  has been solved massively, it immediately raises a question: Is it possible to use those methods when the forcing term dominates? It surely requires highly sophisticated techniques and all surrounding issues need to be invented due to the occurrence of new parts. To the best of our knowledge, rigorous investigation of the instability regime and qualitative analysis of the truncation approach have so far not been considered for the Helmholtz equation with genuinely mixed boundary conditions. Moreover, both the Helmholtz equation and the truncation method are ubiquitous in applied mathematics. It is thus imperative to answer the above question with full details.

Summarizing, our main objectives are
\begin{itemize}
\item proving the underlying model is unstable in the sense of Hadamard and giving a theoretical example for such instability;
\item applying the truncation method to define the regularized solution and showing the error estimates which also imply the stability and strong convergence;
\item providing a short extension of the problem with a nonlinear forcing term.
\end{itemize}

The remainder of this paper is organized as follows. In Section \ref{Sec:setting}, we state the model problem, introduce the abstract settings, and herein discuss thoroughly the nature of ill-posedness. Section \ref{Sec:truncate} is devoted to our second objective whilst the third objective is investigated in Section \ref{Sec:extension}. As a result, the error estimates together with stability are proved with respect
to measurement level and we present explicit formulae for the regularization
parameters. Our analysis is mainly based on the Fourier
transform, superposition principle and Parseval's identity. Interestingly, we observe that when choosing a suitable regularization parameter, the convergence rate stays unchanged from the linear case to the nonlinear case. Numerical tests are provided
in Section \ref{Sec:numerical} to illustrate our  method and Section \ref{Sec:con} concludes the paper with a discussion of our results and forthcoming aims.

\section{Abstract settings and Ill-posedness}\label{Sec:setting}
\subsection{Abstract settings}

Let us consider the problem of reconstructing the radiation
field $u=u(x,y,z)$ in the domain $\Omega=\mathbb{R}^{2}\times\left(0,d\right), d>0$. For simplicity, the first two variables will be denoted by ${\bf \xi}:=(x,y)$. The problem
given by \eqref{eq:1.2} along with the boundary conditions can be
written as follows:
\begin{equation}
\begin{cases}
\begin{array}{llll}
\Delta u+k^{2}u=-f, & & \mbox{in}\;\Omega,\\
u\left({\bf \xi},d\right) = g\left(\xi\right), & & \xi\in\mathbb{R}^{2},\\
\partial_{z}u\left(\xi,d\right) = h\left(\xi\right), & & \xi\in\mathbb{R}^{2},\\
u\left(\cdot,z\right)\in L^{2}\left(\mathbb{R}^{2}\right), & & z\in\left[0,d\right],
\end{array}
\end{cases}\label{eq:helmholtz}
\end{equation}
where $\ensuremath{g,h\in L^{2}\left(\mathbb{R}^{2}\right)}$ are
given data and $\ensuremath{f\in L^{2}\left(\Omega\right)}$ plays
a role as the given forcing term.

In practical applications, it is impossible to express exactly the quantities $f,g$ and $h$ since only measured data are known. Therefore, we aim at looking for an
approximate solution of \eqref{eq:helmholtz} inside $\Omega$.

Let us now specify the following assumptions:

\textbf{(A1)} Let $\ensuremath{\ensuremath{g_{\delta},h_{\delta}\in L^{2}\left(\mathbb{R}^{2}\right)}}$
and $\ensuremath{\ensuremath{f_{\delta}\in L^{2}\left(\Omega\right)}}$
be the measured data with the noise level $\ensuremath{\delta>0}$
such that
\[
\left\Vert f_{\delta}-f\right\Vert _{L^{2}\left(\Omega\right)}\le\delta,\quad\left\Vert g_{\delta}-g\right\Vert _{L^{2}\left(\mathbb{R}^{2}\right)}\le\delta,\quad\left\Vert h_{\delta}-h\right\Vert _{L^{2}\left(\mathbb{R}^{2}\right)}\le\delta;
\]

\textbf{(A2)} The solution $u$ in $H^{2}\left(\Omega\right)$ of
the problem \eqref{eq:helmholtz} does uniquely exist;

\textbf{(A3)} The wave number $k$ and the number $d$ satisfy $kd<\frac{\pi}{2}$.

By the definition of Hadamard, let us recall that a problem is well-posed
if it has a unique solution for all admissible data, and furthermore, dependence of this solution on given data is confirmed. In
this section, we prove that the solution is independent of given data
in a specific case and provide an example that small errors in such data may easily destroy
any numerical solution. This thus illuminates why a regularization method should
be collected and applied to solve this problem.

Due to assumption \textbf{(A2)}, let us consider the solution
$\ensuremath{u}$ of \eqref{eq:helmholtz}, such that it is defined
by the sum of two functions $u_{1}$ and $u_{2}$ where $u_{1}\in H^{2}\left(\Omega\right)$
satisfies
\begin{equation}
\begin{cases}
\begin{array}{llll}
\Delta u_{1}+k^{2}u_{1}=0, & & \mbox{in\;}\Omega,\\
u_{1}\left(\xi,0\right)=0, & & \xi\in\mathbb{R}^{2},\\
\partial_{z}u_{1}\left(\xi,d\right)=h\left(\xi\right), & & \xi\in\mathbb{R}^{2},\\
u_{1}\left(\cdot,z\right)\in L^{2}\left(\mathbb{R}^{2}\right), & & z\in\left[0,d\right],
\end{array}
\end{cases}\label{eq:helmholtz1}
\end{equation}
and $u_{2}\in H^{2}\left(\Omega\right)$ satisfies
\begin{equation}
\begin{cases}
\begin{array}{llll}
\Delta u_{2}+k^{2}u_{2}=-f, & & \mbox{in\;}\Omega,\\
u_{2}\left(\xi,d\right)=g\left(\xi\right)-u_{1}\left(\xi,d\right),& & \xi\in\mathbb{R}^{2},\\
\partial_{z}u_{2}\left(\xi,d\right)=0,& & \xi\in\mathbb{R}^{2},\\
u_{2}\left(\cdot,z\right)\in L^{2}\left(\mathbb{R}^{2}\right), & & z\in\left[0,d\right].
\end{array}
\end{cases}\label{eq:helmholtz2}
\end{equation}

\subsection{Ill-posedness}
In order to analyze the ill-posedness of \eqref{eq:helmholtz}, one
naturally needs to take into consideration  \eqref{eq:helmholtz1}
and \eqref{eq:helmholtz2}. In fact, let us first consider the following
lemma which is proved in \cite{RR06}.
It shows that the assumption \textbf{(A3)} implies that
the solution $u_{1}$ depends continuously on the given data $h$,
i.e. the problem \eqref{eq:helmholtz1} is well-posed. For simplicity,
we present the proof without giving full of details.
\begin{lem}
\label{lem:1} Let $u_{1}$ be the solution of problem \eqref{eq:helmholtz1}, then there exists $C>0$ depending
only on $k$ and $d$ such that
\[
\left\Vert u_{1}\right\Vert _{L^{2}\left(\Omega\right)}\le C\left\Vert h\right\Vert _{L^{2}\left(\mathbb{R}^{2}\right)}.
\]
\end{lem}
\begin{proof}
Since $u_{1}\left(\cdot,z\right)\in H^{2}\left(\mathbb{R}^{2}\right)$
for each $z\in\left[0,d\right]$, we apply the Fourier transform for
two-dimensional cases with respect to variable $\xi=(x,y)\in\mathbb{R}^{2}$
as follows:
\[
\hat{u}_{1}\left(\rho,z\right)=\int_{\mathbb{R}^{2}}u_{1}\left(\xi,z\right)e^{-2\pi i\left\langle \rho,\xi\right\rangle }d\xi,
\]
where $\ensuremath{\rho=\left(\rho_{1},\rho_{2}\right)\in\mathbb{R}^{2}}$
and $\ensuremath{\left\langle \rho,\xi\right\rangle =\rho_{1}x+\rho_{2}y}$.

This function significantly satisfies the problem which is constructed
from \eqref{eq:helmholtz1} in terms of the Fourier transform:
\begin{equation}
\begin{cases}
\begin{array}{llll}
\left(\hat{u}_{1}\right)_{zz}\left(\rho,z\right)=\left(\left|\rho\right|^{2}-k^{2}\right)\hat{u}_{1}\left(\rho,z\right),& & \rho\in\mathbb{R}^{2},z\in\left(0,d\right),\\
\hat{u}_{1}\left(\rho,0\right)=0,& &\rho\in\mathbb{R}^{2},\\
\partial_{z}\hat{u}_{1}\left(\rho,d\right)=\hat{h}\left(\rho\right),& &\rho\in\mathbb{R}^{2}.
\end{array}
\end{cases}\label{eq:fourier1}
\end{equation}

For the sake of simplicity, we define by $\lambda_{\rho,k}:=\left|\rho\right|^{2}-k^{2},$
and set
\begin{equation}
\begin{array}{llllll}
A_{1}:=\left\{\rho\in\mathbb{R}^2:\lambda_{\rho,k}>0\right\} ,\quad & A_{2}:=\left\{ \rho\in\mathbb{R}^2:\lambda_{\rho,k}=0\right\} ,\\
A_{3}:=\left\{\rho\in\mathbb{R}^2:\lambda_{\rho,k}<0\right\} ,\quad & A\;\,:=A_{1}\cup A_{2}\cup A_{3}.
\end{array}\label{eq:A123}
\end{equation}

By direct computation, the solution of problem \eqref{eq:fourier1}
can be represented by
\[
\hat{u}_{1}\left(\rho,z\right)=\left\{ \begin{array}{cc}
{\displaystyle \frac{\hat{h}\left(\rho\right)\sinh\left(z\sqrt{\lambda_{\rho,k}}\right)}{\sqrt{\lambda_{\rho,k}}\cosh\left(d\sqrt{\lambda_{\rho,k}}\right)},} & \mbox{if}\;\rho\in A_{1},\\
z\hat{h}\left(\rho\right), & \mbox{if}\;\rho\in A_{2},\\
{\displaystyle \frac{\hat{h}\left(\rho\right)\sin\left(z\sqrt{-\lambda_{\rho,k}}\right)}{\sqrt{-\lambda_{\rho,k}}\cos\left(d\sqrt{-\lambda_{\rho,k}}\right)},} & \mbox{if}\;\rho\in A_{3}.
\end{array}\right.
\]

Hence, we obtain the desired result with $\ensuremath{C=\sqrt{d}\max\left\{ d,\dfrac{\tan\left(dk\right)}{k}\right\} }$
by the same techniques  studied in \cite{RR06}.

\end{proof}

Thanks to Lemma \ref{lem:1}, the problem \eqref{eq:helmholtz} can be reduced to
the problem \eqref{eq:helmholtz2}. This means that we can study the
problem \eqref{eq:helmholtz} with $h\left(\xi\right)=0$ provided \textbf{(A3)} holds.

In the same strategy, we employ the Fourier transform again to consider the following
problem for \eqref{eq:helmholtz2} and note that for simplicity,
we denote the solution by $u$ instead of $u_{2}$:
\begin{equation}
\begin{cases}
\begin{array}{llll}
\hat{u}_{zz}\left(\rho,z\right)-\lambda_{\rho,k}\hat{u}\left(\rho,z\right)=-\hat{f}\left(\rho,z\right), & & \rho\in\mathbb{R}^{2},z\in\left(0,d\right),\\
\hat{u}\left(\rho,d\right)=\hat{g}\left(\rho\right),& &\rho\in\mathbb{R}^{2},\\
\partial_{z}\hat{u}\left(\rho,d\right)=0, & &\rho\in\mathbb{R}^{2}.
\end{array}
\end{cases}\label{eq:fourier2}
\end{equation}
Based on \eqref{eq:A123}, we first consider the case $\ensuremath{\rho\in A_{1}\cup A_{3}}$.
In this case, the solution to \eqref{eq:fourier2} with respect to
$z$ can be found by superposition principle, namely $u=w_{1}+w_{2}$ is presented,
where  the complementary solution $w_{1}$ satisfies
\begin{equation}
\begin{cases}
\begin{array}{llll}
\hat{u}_{zz}\left(\rho,z\right)-\lambda_{\rho,k}\hat{u}\left(\rho,z\right)=0, & & \rho\in A_{1}\cup A_{3},z\in\left(0,d\right),\\
\hat{u}\left(\rho,d\right)=\hat{g}\left(\rho\right), & &\\
\partial_{z}\hat{u}\left(\rho,d\right)=0, & &,
\end{array}
\end{cases}\label{eq:fourier21}
\end{equation}
and  the particular solution $w_{2}$ satisfies
\begin{equation}
\begin{cases}
\begin{array}{llll}
\hat{u}_{zz}\left(\rho,z\right)-\lambda_{\rho,k}\hat{u}\left(\rho,z\right)=-\hat{f}\left(\rho,z\right),& & \rho\in A_{1}\cup A_{3},z\in\left(0,d\right),\\
\hat{u}\left(\rho,d\right)=0,& &\\
\partial_{z}\hat{u}\left(\rho,d\right)=0, & &.
\end{array}
\end{cases}\label{eq:fourier22}
\end{equation}

We now start computing these two solutions by the following lemmas.

\begin{lem}
\label{lem:2}The solution $w_{1}$ to \eqref{eq:fourier21} has the
form
\begin{equation}
w_{1}\left(\rho,z\right)=\begin{cases}
\hat{g}\left(\rho\right)\cosh\left(\left(d-z\right)\sqrt{\lambda_{\rho,k}}\right), & \rho\in A_{1},\\
\hat{g}\left(\rho\right)\cos\left(\left(d-z\right)\sqrt{-\lambda_{\rho,k}}\right), & \rho\in A_{3}.
\end{cases}\label{eq:w1}
\end{equation}
\end{lem}
\begin{proof}
If $\ensuremath{\rho\in A_{1}}$, solving the ordinary differential equation \eqref{eq:fourier21} yields that
\begin{equation}
w_{1}\left(\rho,z\right)=C_{1}\left(\rho\right)e^{z\sqrt{\lambda_{\rho,k}}}+C_{2}\left(\rho\right)e^{-z\sqrt{\lambda_{\rho,k}}}.\label{eq:2.10-1}
\end{equation}
Then, its derivative obeys the relation
\begin{equation}
\frac{\partial w_{1}}{\partial z}\left(\rho,z\right)=\sqrt{\lambda_{\rho,k}}\left[C_{1}\left(\rho\right)e^{z\sqrt{\lambda_{\rho,k}}}-C_{2}\left(\rho\right)e^{-z\sqrt{\lambda_{\rho,k}}}\right].\label{eq:2.10}
\end{equation}
Due to the fact that $\partial_{z}w_{1}\left(\rho,d\right)=0$, it is easy to point
out from \eqref{eq:2.10} that $C_{2}\left(\rho\right)=C_{1}\left(\rho\right)e^{2d\sqrt{\lambda_{\rho,k}}}$.
Therefore, substituting  this quantity into \eqref{eq:2.10-1} we obtain 
\begin{equation}
w_{1}\left(\rho,z\right)=C_{1}\left(\rho\right)\left[e^{z\sqrt{\lambda_{\rho,k}}}+e^{\left(2d-z\right)\sqrt{\lambda_{\rho,k}}}\right].\label{eq:2.12}
\end{equation}
Now it remains to consider the condition $\ensuremath{w_{1}\left(\rho,d\right)=\hat{g}\left(\rho\right)}$.
We get $\ensuremath{C_{1}\left(\rho\right)=\dfrac{1}{2}\hat{g}\left(\rho\right)e^{-d\sqrt{\lambda_{\rho,k}}}}$
and it follows from \eqref{eq:2.12} that
\begin{equation}
w_{1}\left(\rho,z\right)=\frac{1}{2}\hat{g}\left(\rho\right)\left[e^{\left(z-d\right)\sqrt{\lambda_{\rho,k}}}+e^{\left(d-z\right)\sqrt{\lambda_{\rho,k}}}\right]=\hat{g}\left(\rho\right)\cosh\left(\left(d-z\right)\sqrt{\lambda_{\rho,k}}\right).\label{eq:2.13}
\end{equation}

For the case $\ensuremath{\rho\in A_{3}}$,
we do the same vein and obtain
\begin{equation}
w_{1}\left(\rho,z\right)=\hat{g}\left(\rho\right)\cos\left(\left(d-z\right)\sqrt{-\lambda_{\rho,k}}\right).\label{eq:2.14}
\end{equation}
Combining \eqref{eq:2.13} and \eqref{eq:2.14}, we complete the proof
of the lemma.\end{proof}
\begin{lem}
\label{lem:3}The solution $w_{2}$ to \eqref{eq:fourier22} has the
form
\[
w_{2}\left(\rho,z\right)=\begin{cases}
\left(\sqrt{\lambda_{\rho,k}}\right)^{-1}\displaystyle{\int_{z}^{d}}\hat{f}\left(\rho,s\right)\sinh\left(\left(z-s\right)\sqrt{\lambda_{\rho,k}}\right)ds, & \rho\in A_{1},\\
\left(\sqrt{-\lambda_{\rho,k}}\right)^{-1}\displaystyle{\int_{z}^{d}}\hat{f}\left(\rho,s\right)\sin\left(\left(z-s\right)\sqrt{-\lambda_{\rho,k}}\right)ds, & \rho\in A_{3}.
\end{cases}
\]
\end{lem}
\begin{proof}
The particular solution $w_{2}$ in the case $\rho\in A_{1}$
is
\begin{equation}
w_{2}\left(\rho,z\right)=F\left(\rho,z\right)e^{z\sqrt{\lambda_{\rho,k}}}+G\left(\rho,z\right)e^{-z\sqrt{\lambda_{\rho,k}}}.\label{eq:2.15-1}
\end{equation}
In order to compute $F$ and $G$, one needs to solve the following
system:
\begin{equation}
\begin{cases}
\frac{\partial F}{\partial z}\left(\rho,z\right)e^{z\sqrt{\lambda_{\rho,k}}}+\frac{\partial G}{\partial z}\left(\rho,z\right)e^{-z\sqrt{\lambda_{\rho,k}}}=0,\\
\frac{\partial F}{\partial z}\left(\rho,z\right)e^{z\sqrt{\lambda_{\rho,k}}}-\frac{\partial G}{\partial z}\left(\rho,z\right)e^{-z\sqrt{\lambda_{\rho,k}}}=-\left(\sqrt{\lambda_{\rho,k}}\right)^{-1}\hat{f}\left(\rho,z\right).
\end{cases}\label{eq:2.15}
\end{equation}
By using elementary computation, the system \eqref{eq:2.15} is thus
equivalent to
\begin{equation}
\begin{cases}
\frac{\partial F}{\partial z}\left(\rho,z\right)=-\frac{1}{2}\left(\sqrt{\lambda_{\rho,k}}\right)^{-1}\hat{f}\left(\rho,z\right)e^{-z\sqrt{\lambda_{\rho,k}}},\\
\frac{\partial G}{\partial z}\left(\rho,z\right)=\frac{1}{2}\left(\sqrt{\lambda_{\rho,k}}\right)^{-1}\hat{f}\left(\rho,z\right)e^{z\sqrt{\lambda_{\rho,k}}}.
\end{cases}\label{eq:2.16}
\end{equation}
It follows from \eqref{eq:2.16} that
\begin{equation}
F\left(\rho,z\right)=-\frac{1}{2}\left(\sqrt{\lambda_{\rho,k}}\right)^{-1}\int_{0}^{z}\hat{f}\left(\rho,s\right)e^{-s\sqrt{\lambda_{\rho,k}}}ds+C_{1}\left(\rho\right),\label{eq:F}
\end{equation}
\begin{equation}
G\left(\rho,z\right)=\frac{1}{2}\left(\sqrt{\lambda_{\rho,k}}\right)^{-1}\int_{0}^{z}\hat{f}\left(\rho,s\right)e^{s\sqrt{\lambda_{\rho,k}}}ds+C_{2}\left(\rho\right).\label{eq:G}
\end{equation}
Moreover, from \eqref{eq:2.15-1} we have
\begin{equation}
\frac{\partial w_{2}}{\partial z}\left(\rho,z\right)=\left[\frac{\partial F}{\partial z}\left(\rho,z\right)+\sqrt{\lambda_{\rho,k}}F\left(\rho,z\right)\right]e^{z\sqrt{\lambda_{\rho,k}}}+\left[\frac{\partial G}{\partial z}\left(\rho,z\right)-\sqrt{\lambda_{\rho,k}}G\left(\rho,z\right)\right]e^{-z\sqrt{\lambda_{\rho,k}}}.\label{eq:2.20}
\end{equation}
We plug the boundary conditions $w_{2}\left(\rho,d\right)=\partial_{z}w_{2}\left(\rho,d\right)=0$ into \eqref{eq:2.15-1} and \eqref{eq:2.20} and
obtain the following system:
\[
\begin{cases}
F\left(\rho,d\right)e^{d\sqrt{\lambda_{\rho,k}}}+G\left(\rho,d\right)e^{-d\sqrt{\lambda_{\rho,k}}}=0,\\
F\left(\rho,d\right)e^{d\sqrt{\lambda_{\rho,k}}}-G\left(\rho,d\right)e^{-d\sqrt{\lambda_{\rho,k}}}=0.
\end{cases}
\]
As a consequence, the terms $C_{1}\left(\rho\right)$ and $C_{2}\left(\rho\right)$
in \eqref{eq:F}-\eqref{eq:G} can be determined by
\[
C_{1}\left(\rho\right)=\frac{1}{2}\left(\sqrt{\lambda_{\rho,k}}\right)^{-1}\int_{0}^{d}\hat{f}\left(\rho,s\right)e^{-s\sqrt{\lambda_{\rho,k}}}ds,
\]
\[
C_{2}\left(\rho\right)=-\frac{1}{2}\left(\sqrt{\lambda_{\rho,k}}\right)^{-1}\int_{0}^{d}\hat{f}\left(\rho,s\right)e^{s\sqrt{\lambda_{\rho,k}}}ds.
\]
Combining this with \eqref{eq:F}-\eqref{eq:G} and \eqref{eq:2.15-1},
the solution $w_{2}$ is formed by
\begin{equation}
w_{2}\left(\rho,z\right)=\frac{1}{2}\left(\sqrt{\lambda_{\rho,k}}\right)^{-1}\int_{z}^{d}\hat{f}\left(\rho,s\right)\left(e^{\left(z-s\right)\sqrt{\lambda_{\rho,k}}}-e^{\left(s-z\right)\sqrt{\lambda_{\rho,k}}}\right)ds=\left(\sqrt{\lambda_{\rho,k}}\right)^{-1}\int_{z}^{d}\hat{f}\left(\rho,s\right)\sinh\left(\left(z-s\right)\sqrt{\lambda_{\rho,k}}\right)ds.\label{eq:2.21}
\end{equation}

For the case $\ensuremath{\rho\in A_{3}}$,
we also do the same manner and obtain
\begin{equation}
w_{2}\left(\rho,z\right)=\left(\sqrt{-\lambda_{\rho,k}}\right)^{-1}\int_{z}^{d}\hat{f}\left(\rho,s\right)\sin\left(\left(z-s\right)\sqrt{-\lambda_{\rho,k}}\right)ds.\label{eq:2.22}
\end{equation}
Hence, \eqref{eq:2.21}-\eqref{eq:2.22} give the desired result.
\end{proof}
In Lemma \ref{lem:2}-\ref{lem:3}, the solution of the problem
\eqref{eq:fourier2} for $\rho\in A_{1}\cup A_{3}$
has been formulated. It remains to find the solution when the case
$A_{2}$ happens. Let us show it in the following lemma.

\begin{lem}
\label{lem:4}Let $w$ be the solution of the following problem
\[
\begin{cases}
\begin{array}{llll}
\hat{u}_{zz}\left(\rho,z\right)=-\hat{f}\left(\rho,z\right),& &\rho\in\mathbb{R}^{2},z\in\left(0,d\right),\\
\hat{u}\left(\rho,d\right)=\hat{g}\left(\rho\right),& &\rho\in\mathbb{R}^{2},\\
\partial_{z}\hat{u}\left(\rho,d\right)=0,& &\rho\in\mathbb{R}^{2}.
\end{array}
\end{cases}
\]
Then, we have
\[
w\left(\rho,z\right)={\displaystyle \hat{g}\left(\rho\right)+\int_{d}^{z}\int_{s}^{d}\hat{f}\left(\rho,\gamma\right)d\gamma ds}.
\]
\end{lem}
\begin{proof}
The proof is straightforward. Indeed, using the Newton-Leibniz formula
twice, together with the boundary condition $\partial_{z}\hat{u}\left(\rho,d\right)=0$, we are led to the following equalities:
\[
\partial_{z}\hat{u}\left(\rho,z\right)=\partial_{z}\hat{u}\left(\rho,d\right)+\int_{z}^{d}\hat{f}\left(\rho,s\right)ds=\int_{z}^{d}\hat{f}\left(\rho,s\right)ds.
\]
Therefore, the solution can be expressed as follows:
\[
\hat{u}\left(\rho,z\right)=\hat{u}\left(\rho,d\right)-\int_{z}^{d}\partial_{z}\hat{u}\left(\rho,s\right)ds=\hat{g}\left(\rho\right)+\int_{d}^{z}\int_{s}^{d}\hat{f}\left(\rho,\gamma\right)d\gamma ds.
\]
This completes the proof of the lemma.
\end{proof}

In principle, we have proved the representation of the solution to
the problem \eqref{eq:fourier2}. It is given by
\begin{equation}
\hat{u}\left(\rho,z\right)=\left\{ \begin{array}{cc}
{\displaystyle \hat{g}\left(\rho\right)\cosh\left(\left(d-z\right)\sqrt{\lambda_{\rho,k}}\right)}\\
{\displaystyle +\frac{1}{\sqrt{\lambda_{\rho,k}}}\int_{z}^{d}\hat{f}\left(\rho,s\right)\sinh\left(\left(z-s\right)\sqrt{\lambda_{\rho,k}}\right)ds,} & \mbox{if}\;\rho\in A_{1},\\
{\displaystyle \hat{g}\left(\rho\right)+\int_{d}^{z}\int_{s}^{d}\hat{f}\left(\rho,\gamma\right)d\gamma ds,} & \mbox{if}\;\rho\in A_{2},\\
{\displaystyle {\displaystyle \hat{g}\left(\rho\right)\cos\left(\left(d-z\right)\sqrt{-\lambda_{\rho,k}}\right)}}\\
{\displaystyle +\frac{1}{\sqrt{-\lambda_{\rho,k}}}\int_{z}^{d}\hat{f}\left(\rho,s\right)\sin\left(\left(z-s\right)\sqrt{-\lambda_{\rho,k}}\right)ds,} & \mbox{if}\;\rho\in A_{3}.
\end{array}\right.\label{eq:uhat}
\end{equation}

The sufficient condition for the existence of the Fourier transform
of a function is guaranteed if this function is absolutely integrable. Thus, the regularity
of given functions $g,h$ and $f$ are definitely useful. Moreover, to recover the value of a function
at a point from its inverse transform, a bounded variation
in a neighborhood of that point is needed. In addition, the Fourier
transform of a $L^{2}$-function is guaranteed to be another $L^{2}$-function. Hereby, the integral representation of the explicit solution
$\hat{u}\left(\rho,z\right)$ of the problem \eqref{eq:fourier2}
does exist.

The important thing observed from \eqref{eq:uhat} is that, the terms $\cosh\left((d-z)\sqrt{\lambda_{\rho,k}}\right)$
and $\sinh\left((z-s)\sqrt{\lambda_{\rho,k}}\right)$
include increasing terms $e^{(d-z)\sqrt{\lambda_{\rho,k}}}$ and
$e^{(s-z)\sqrt{\lambda_{\rho,k}}}$, respectively, for $z\le s\le d$. 
Notice that even if the exact Fourier coefficients $\hat{f}$ and $\hat{g}$ may tend to zero
rapidly, computational procedures are in general impossible. A small
perturbation, such as round-off errors, in the data may exacerbate a large
error in the solution. For the sake of clarity, we shall prove below
that the problem \eqref{eq:fourier2} (also the problem \eqref{eq:helmholtz})
is ill-posed in the case $A_{1}$ while it is observed to be well-posed
in the case $A_{2}\cup A_{3}$.

\begin{lem}
If $w$ is the solution of problem \eqref{eq:fourier2} when $\rho\in A_{2}\cup A_{3}$, it depends continuously
on $g$ and $f$ in the sense that
\[
\left\Vert w\right\Vert _{L^{2}\left(\Omega\right)}^{2}\le C\left(\left\Vert g\right\Vert _{L^{2}\left(\mathbb{R}^{2}\right)}^{2}+\left\Vert f\right\Vert _{L^{2}\left(\Omega\right)}^{2}\right),
\]
where $C$ is a positive constant depending only on $d$ and $k$.\end{lem}
\begin{proof}
The current strategy is completely straightforward since the representation
has been given in \eqref{eq:uhat}.

If $\rho\in A_{2}$,  applying H\"{o}lder's inequality and Lemma \ref{lem:4}, we have:
\[
\left|w\left(\rho,z\right)\right|\le\left|\hat{g}\left(\rho\right)\right|+\int_{z}^{d}\left(d-s\right)\left(\int_{s}^{d}\left|\hat{f}\left(\rho,\gamma\right)\right|^{2}d\gamma\right)^{1/2}ds\le\left|\hat{g}\left(\rho\right)\right|+d\int_{z}^{d}\left(\int_{s}^{d}\left|\hat{f}\left(\rho,\gamma\right)\right|^{2}d\gamma\right)^{1/2}ds.
\]
Thus,  the Cauchy-Schwarz inequality and H\"{o}lder's inequality
imply
\begin{eqnarray}
\int_{\mathbb{R}^{2}}\left|w\left(\rho,z\right)\right|^{2}d\rho & \le & 2\int_{\mathbb{R}^{2}}\left[\left|\hat{g}\left(\rho\right)\right|^{2}+d^{2}\left(\int_{z}^{d}\left(\int_{s}^{d}\left|\hat{f}\left(\rho,\gamma\right)\right|^{2}d\gamma\right)^{1/2}ds\right)^{2}ds\right]d\rho\nonumber \\
 & \le & 2\left[\left\Vert g\right\Vert _{L^{2}\left(\mathbb{R}^{2}\right)}^{2}+d^{2}\left(d-z\right)\int_{\mathbb{R}^{2}}\int_{z}^{d}\int_{s}^{d}\left|\hat{f}\left(\rho,\gamma\right)\right|^{2}d\gamma dsd\rho\right]\nonumber \\
 & \le & 2\left[\left\Vert g\right\Vert _{L^{2}\left(\mathbb{R}^{2}\right)}^{2}+d^{3}\int_{\mathbb{R}^{2}}\int_{z}^{d}\int_{s}^{d}\left|\hat{f}\left(\rho,\gamma\right)\right|^{2}d\gamma dsd\rho\right]\nonumber \\
 & \le & 2\left[\left\Vert g\right\Vert _{L^{2}\left(\mathbb{R}^{2}\right)}^{2}+d^{3}\left(d-z\right)\left\Vert f\right\Vert _{L^{2}\left(\Omega\right)}^{2}\right].\label{eq:2.24}
\end{eqnarray}

If $\rho\in A_{3}$,  the assumption \textbf{(A3)} gives us the following inequalities:
\begin{equation}
0<\left(s-z\right)\sqrt{-\lambda_{\rho,k}}\le\left(d-z\right)\sqrt{-\lambda_{\rho,k}}\le dk<\frac{\pi}{2},\quad z\le s\le d,\label{eq:2.25}
\end{equation}
which easily leads to
\begin{equation}
\left|\sin\left(\left(z-s\right)\sqrt{-\lambda_{\rho,k}}\right)\right|\le\sin\left(\left(s-z\right)\sqrt{-\lambda_{\rho,k}}\right)\le\sin\left(\left(d-z\right)\sqrt{-\lambda_{\rho,k}}\right),\label{eq:2.26}
\end{equation}

\begin{equation}
\tan\left(\left(d-z\right)\sqrt{-\lambda_{\rho,k}}\right)\le\tan\left(d\sqrt{-\lambda_{\rho,k}}\right).\label{eq:2.27}
\end{equation}
In addition, thanks to the fact that the function $\dfrac{\tan x}{x}$ is increasing for $x\in\left(0,\dfrac{\pi}{2}\right)$,
we combine this with \eqref{eq:uhat} and \eqref{eq:2.25}-\eqref{eq:2.27} to  obtain
\begin{eqnarray}
\left|w\left(\rho,z\right)\right| & \le & \cos\left(\left(d-z\right)\sqrt{-\lambda_{\rho,k}}\right)\left|\hat{g}\left(\rho\right)\right|+\frac{1}{\sqrt{-\lambda_{\rho,k}}}\int_{z}^{d}\left|\hat{f}\left(\rho,s\right)\right|\sin\left(\left(d-z\right)\sqrt{-\lambda_{\rho,k}}\right)ds\nonumber \\
 & \le & \left|\hat{g}\left(\rho\right)\right|+\frac{\tan\left(\left(d-z\right)\sqrt{-\lambda_{\rho,k}}\right)}{\sqrt{-\lambda_{\rho,k}}}\int_{z}^{d}\left|\hat{f}\left(\rho,s\right)\right|ds\nonumber \\
 & \le & \left|\hat{g}\left(\rho\right)\right|+\frac{\tan\left(d\sqrt{-\lambda_{\rho,k}}\right)}{\sqrt{-\lambda_{\rho,k}}}\int_{z}^{d}\left|\hat{f}\left(\rho,s\right)\right|ds\nonumber \\
 & \le & \left|\hat{g}\left(\rho\right)\right|+\frac{\tan\left(dk\right)}{k}\int_{0}^{d}\left|\hat{f}\left(\rho,s\right)\right|ds\nonumber \\
 & \le & C\left(\left|\hat{g}\left(\rho\right)\right|+\int_{0}^{d}\left|\hat{f}\left(\rho,s\right)\right|ds\right),\label{eq:2.28-1}
\end{eqnarray}
where we have put $C=\max\left\{ \dfrac{\tan\left(dk\right)}{k},1\right\} $. Next, it follows from H\"{o}lder's inequality that
\begin{equation}
\int_{0}^{d}\left|\hat{f}\left(\rho,s\right)\right|ds\le\sqrt{d}\left(\int_{0}^{d}\left|\hat{f}\left(\rho,s\right)\right|^{2}ds\right)^{1/2}.\label{eq:2.28}
\end{equation}
Applying  the Cauchy-Schwarz inequality and  substituting \eqref{eq:2.28}
into \eqref{eq:2.28-1}, we arrive at
\begin{eqnarray*}
\int_{\mathbb{R}^{2}}\left|w\left(\rho,z\right)\right|^{2}d\rho & \le & C^{2}\int_{\mathbb{R}^{2}}\left(\left|\hat{g}\left(\rho\right)\right|+\int_{0}^{d}\left|\hat{f}\left(\rho,s\right)\right|ds\right)^{2}d\rho\\
 & \le & 2C^{2}\left[\int_{\mathbb{R}^{2}}\left|\hat{g}\left(\rho\right)\right|^{2}d\rho+d\int_{\mathbb{R}^{2}}\int_{0}^{d}\left|\hat{f}\left(\rho,s\right)\right|^{2}dsd\rho\right]\\
 & \le & 2C^{2}\left(\left\Vert g\right\Vert _{L^{2}\left(\mathbb{R}^{2}\right)}^{2}+d\left\Vert f\right\Vert _{L^{2}\left(\Omega\right)}^{2}\right).
\end{eqnarray*}
Hence, we obtain
\[
\left\Vert w\right\Vert _{L^{2}\left(\Omega\right)}^{2}\le2dC^{2}\left(\left\Vert g\right\Vert _{L^{2}\left(\mathbb{R}^{2}\right)}^{2}+d\left\Vert f\right\Vert _{L^{2}\left(\Omega\right)}^{2}\right),
\]
which ends the proof of the lemma.\end{proof}
\begin{lem}
\label{lem:6}The problem \eqref{eq:fourier2} in the case $\rho\in A_{1}$
is ill-posed.\end{lem}
\begin{proof}
To show the instability of $u$ in this case, the idea is that we
construct two functions $g_{n}(x,y)$ and $f_{n}(x,y,z)$ for $n\in\mathbb{N}$
defined by the Fourier transform, as follows:
\begin{equation}
\widehat{g_n}(\rho):=\widehat{g_n}(\rho_{1},\rho_{2})=\left\{ \begin{array}{llll}
\sqrt{n}, & \mbox{if}\;~\rho\in W_{n},\\
0, & \mbox{if}\;~\rho\in\mathbb{R}^{2}\backslash W_{n},
\end{array}\right.\label{eq:ghat1}
\end{equation}
\begin{equation}
\widehat{f_n}(\rho,s):=\frac{1}{d}\widehat{g_n}(\rho)~~\text{for all}~~s\in[0,d],\label{eq:hhat1}
\end{equation}
where  $W_{n}\subset\mathbb{R}^{2}$ denoted  by
\[
W_{n}:=\Bigg\{\rho=(\rho_{1},\rho_{2})\in\mathbb{R}^{2}|~~n+k+1<\rho_{1},\rho_{2}<n+k+1+\frac{1}{n}\Bigg\}.
\]
By Parseval's identity, it follows from \eqref{eq:ghat1} and \eqref{eq:hhat1}
that
\begin{equation}
\|g_{n}\|_{L^{2}(\mathbb{R}^{2})}^{2}=\int_{\mathbb{R}^{2}}\Big|\widehat{g_n}(\rho)\Big|^{2}d\rho=\int_{n+k+1}^{n+k+1+\frac{1}{n}}\int_{n+k+1}^{n+k+1+\frac{1}{n}}\Big|\sqrt{n}\Big|^{2}d\rho_{1}d\rho_{2}=\frac{1}{n},\label{eq:2.32-1}
\end{equation}
\begin{equation}
\|f_{n}(\cdot,s)\|_{L^{2}(\mathbb{R}^{2})}^{2}=\frac{1}{d}\|g_{n}\|_{L^{2}(\mathbb{R}^{2})}^{2}=\frac{1}{nd}.\label{eq:2.33-1}
\end{equation}
Furthermore, the representation of $\hat{u}$ in \eqref{eq:uhat} reveals
\[
\widehat{u_{n}}\left(\rho,0\right)=\widehat{g_{n}}\left(\rho\right)\cosh\left(d\sqrt{\lambda_{\rho,k}}\right)+\frac{1}{\sqrt{\lambda_{\rho,k}}}\int_{0}^{d}\widehat{f_{n}}\left(\rho,s\right)\sinh\left(-s\sqrt{\lambda_{\rho,k}}\right)ds.
\]
Using the following  simple inequalities,
\[
\sinh\left(-s\sqrt{\lambda_{\rho,k}}\right)\le\frac{e^{s\sqrt{\lambda_{\rho,k}}}}{2}\le\frac{e^{d\sqrt{\lambda_{\rho,k}}}}{2}\le\cosh\left(d\sqrt{\lambda_{\rho,k}}\right),
\]
we thus have
\[
\left|\frac{1}{\sqrt{\lambda_{\rho,k}}}\int_{0}^{d}\widehat{f_{n}}\left(\rho,s\right)\sinh\left(-s\sqrt{\lambda_{\rho,k}}\right)ds\right|\le\left|\frac{1}{\sqrt{\lambda_{\rho,k}}}\int_{0}^{d}\frac{1}{d}\widehat{g_n}(\rho)\cosh\left(d\sqrt{\lambda_{\rho,k}}\right)ds\right|=\frac{\left|\widehat{g_{n}}\left(\rho\right)\cosh\left(d\sqrt{\lambda_{\rho,k}}\right)\right|}{\sqrt{\lambda_{\rho,k}}}.
\]
Therefore, for $\rho\in A_{1}$, it yields that
\[
\int_{\mathbb{R}^{2}}\left|\frac{1}{\sqrt{\lambda_{\rho,k}}}\int_{0}^{d}\widehat{f_{n}}\left(\rho,s\right)\sinh\left(-s\sqrt{\lambda_{\rho,k}}\right)ds\right|^{2}d\rho\le\int_{W_{n}}\frac{\left|\widehat{g_{n}}\left(\rho\right)\cosh\left(d\sqrt{\lambda_{\rho,k}}\right)\right|^{2}}{\left|\rho\right|^{2}-k^{2}}d\rho\le\int_{W_{n}}\frac{\left|\widehat{g_{n}}\left(\rho\right)\cosh\left(d\sqrt{\lambda_{\rho,k}}\right)\right|^{2}}{n^{2}+1}d\rho.
\]
Combining this with the following elementary inequality:
\[
\int_{\mathbb{R}^{2}}|E(\rho)+F(\rho)|^{2}d\rho\ge\frac{1}{2}\int_{\mathbb{R}^{2}}|E(\rho)|^{2}d\rho-\int_{\mathbb{R}^{2}}|F(\rho)|^{2}d\rho,
\]
and using Parseval's identity, we thus obtain
\begin{eqnarray}
\|u_{n}(\cdot,\cdot,0)\|_{L^{2}(\mathbb{R}^{2})}^{2} & = & \int_{\mathbb{R}^{2}}\left|\widehat{u_{n}}\left(\rho,0\right)\right|^{2}d\rho\nonumber \\
 & \ge & \frac{1}{2}\int_{\mathbb{R}^{2}}\left|\widehat{g_{n}}\left(\rho\right)\cosh\left(d\sqrt{\lambda_{\rho,k}}\right)\right|^{2}d\rho-\int_{\mathbb{R}^{2}}\left|\frac{1}{\sqrt{\lambda_{\rho,k}}}\int_{0}^{d}\widehat{f_{n}}\left(\rho,s\right)\sinh\left(-s\sqrt{\lambda_{\rho,k}}\right)ds\right|^{2}d\rho\nonumber \\
 & \ge & \frac{1}{2}\int_{W_{n}}\left|\widehat{g_{n}}\left(\rho\right)\cosh\left(d\sqrt{\lambda_{\rho,k}}\right)\right|^{2}d\rho-\int_{W_{n}}\frac{\left|\widehat{g_{n}}\left(\rho\right)\cosh\left(d\sqrt{\lambda_{\rho,k}}\right)\right|^{2}}{n^{2}+1}d\rho\nonumber \\
 & = & \frac{n^{2}-1}{2(n^{2}+1)}\int_{W_{n}}\left|\widehat{g_{n}}\left(\rho\right)\cosh\left(d\sqrt{\lambda_{\rho,k}}\right)\right|^{2}d\rho.\label{eq:2.32}
\end{eqnarray}
Moreover, the following estimate:
\[
\cosh^{2}\left(d\sqrt{\lambda_{\rho,k}}\right)>\frac{1}{4}e^{2d\sqrt{\lambda_{\rho,k}}}>\frac{e^{2\left|\rho\right|d}}{4e^{2kd}}
\]
 implies that
\begin{eqnarray}
\int_{W_{n}}\left|\widehat{g_{n}}\left(\rho\right)\cosh\left(d\sqrt{\lambda_{\rho,k}}\right)\right|^{2}d\rho & = & \int_{n+k+1}^{n+k+1+\frac{1}{n}}\int_{n+k+1}^{n+k+1+\frac{1}{n}}\Big|\sqrt{n}\Big|^{2}\cosh^{2}\left(d\sqrt{\lambda_{\rho,k}}\right)d\rho_{1}d\rho_{2}\nonumber \\
 & > & \frac{2n}{4e^{2kd}}\int_{n+k+1}^{n+k+1+\frac{1}{n}}\int_{n+k+1}^{n+k+1+\frac{1}{n}}e^{2\left|\rho\right|d}d\rho_{1}d\rho_{2}\nonumber \\
 & > & \frac{2n}{4e^{2kd}}e^{2nd}\int_{n+k+1}^{n+k+1+\frac{1}{n}}\int_{n+k+1}^{n+k+1+\frac{1}{n}}d\rho_{1}d\rho_{2}=\frac{e^{2dn}}{4ne^{2kd}}.\label{eq:2.33}
\end{eqnarray}
Combining \eqref{eq:2.32} and \eqref{eq:2.33}, we get 
\begin{equation}
\|u_{n}(\cdot,\cdot,0)\|_{L^{2}(\mathbb{R}^{2})}^{2}>\frac{n^{2}-1}{2(n^{2}+1)}\frac{e^{2dn}}{4ne^{2kd}}.\label{eq:2.36}
\end{equation}
Notice that \eqref{eq:2.32-1} and \eqref{eq:2.33-1} mean
\[
\lim_{n\to\infty}\|g_{n}\|_{L^{2}(\mathbb{R}^{2})}^{2}=\lim_{n\to\infty}\|f_{n}(\cdot,s)\|_{L^{2}(\mathbb{R}^{2})}^{2}=0,
\]
whereas it follows from \eqref{eq:2.36} that
\[
\lim_{n\to\infty}\|u_{n}(\cdot,\cdot,0)\|_{L^{2}(\mathbb{R}^{2})}^{2}=\infty.
\]
Thus, the problem \eqref{eq:fourier2} is, in general,
ill-posed in the Hadamard sense in $L^{2}$-norm. 
\end{proof}
We notably mention that the ill-posedness of the Helmholtz equation with Cauchy data has been
discussed in several manuscripts, for example \cite{QWS09,RR06,Xiong10}. However, to the best of our knowledge,
most of the works until now did not give
any theoretical result that would prove such ill-posedness like Lemma \ref{lem:6}.
In particular, this is the first time the ill-posedness of the problem considered in this paper is proved. Moreover, our results seem to significantly extend dozens of papers by considering the forcing term $f$. Besides, the appearance of the measured forcing term $f_{\delta}$ makes the regularization procedure rather difficult and requires highly sophisticated techniques later on.

\section{The truncation method}\label{Sec:truncate}

The problem \eqref{eq:helmholtz} has been
proved to be ill-posed when $A_1$ happens. In this section, we shall use the truncation method to stabilize this
problem by constructing a regularized solution. Motivated by the
results discussed in Section \ref{Sec:setting}, we replace the measured data by their truncated functions. These truncated functions are
solely different from zero in a bounded set controlled and parameterized by the
so-called regularization parameter $\varepsilon>0$ depending on the noise
level $\delta$. More precisely, for any fixed $z$, we define by
\[
\begin{bmatrix}\hat{f}_{\delta}^{\varepsilon}\left(\rho,z\right) & \hat{g}_{\delta}^{\varepsilon}\left(\rho\right)\end{bmatrix}:=\begin{cases}
\begin{bmatrix}\hat{f}_{\delta}\left(\rho,z\right) & \hat{g}_{\delta}\left(\rho\right)\end{bmatrix}, & \rho\in\Theta_{\varepsilon},\\
[0\quad 0], & \rho\notin\Theta_{\varepsilon},
\end{cases}
\]
where the bounded set $\Theta_{\varepsilon}$ is defined by
\[
\Theta_{\varepsilon}:=\left\{ \rho\in\mathbb{R}^{2}:\left|\rho\right|^{2}\le\frac{1}{\varepsilon}\right\},
\]
and the functions $\hat{f}_{\delta}\left(\rho,z\right)$ and $\hat{g}_{\delta}\left(\rho\right)$ are the Fourier transforms of the measured data $f_{\delta}$ and $g_{\delta}$ respectively:
\[
\hat{f}_{\delta}\left(\rho,z\right)=\int_{\mathbb{R}^{2}}f_{\delta}\left(\xi,z\right)e^{-2\pi i\left\langle \rho,\xi\right\rangle }d\xi,\quad\hat{g}_{\delta}\left(\rho\right)=\int_{\mathbb{R}^{2}}g_{\delta}\left(\xi\right)e^{-2\pi i\left\langle \rho,\xi\right\rangle }d\xi.
\]
Here we recall that $\rho=\left(\rho_{1},\rho_{2}\right)\in\mathbb{R}^{2}$, $\xi=(x,y)\in\mathbb{R}^2$ and
$\left\langle \rho,\xi\right\rangle =\rho_{1}x+\rho_{2}y$. 

In the same manner, the truncated functions of the exact data are defined by
\[\begin{bmatrix}\hat{f}^{\varepsilon}\left(\rho,z\right) & \hat{g}^{\varepsilon}\left(\rho\right)\end{bmatrix}:=\begin{cases}
\begin{bmatrix}\hat{f}\left(\rho,z\right) & \hat{g}\left(\rho\right)\end{bmatrix}, & \rho\in\Theta_{\varepsilon},\\
[0\quad 0], & \rho\notin\Theta_{\varepsilon},
\end{cases}
\]

Now we are in a position to introduce the regularized solution $\hat{u}_{\delta}^{\varepsilon}$ to \eqref{eq:fourier2} corresponding to the noise level $\delta$ and the regularization parameter $\varepsilon$. Notice that it is sufficient to restrict to the case $\rho\in A_1$. The regularized solution $\hat{u}_{\delta}^{\varepsilon}$ of the problem \eqref{eq:fourier2} along with the
boundary condition $\hat{u}_{\delta}^{\varepsilon}\left(\rho,d\right)=\hat{g}_{\delta}^{\varepsilon}\left(\rho\right)$
and  the forcing
function $-\hat{f}_{\delta}^{\varepsilon}\left(\rho,z\right)$ reads:
\begin{equation}
\hat{u}_{\delta}^{\varepsilon}\left(\rho,z\right):=\hat{g}_{\delta}^{\varepsilon}\left(\rho\right)\cosh\left(\left(d-z\right)\sqrt{\lambda_{\rho,k}}\right)+\frac{1}{\sqrt{\lambda_{\rho,k}}}\int_{z}^{d}\hat{f}_{\delta}^{\varepsilon}\left(\rho,s\right)\sinh\left(\left(z-s\right)\sqrt{\lambda_{\rho,k}}\right)ds.\label{eq:uhatdelta}
\end{equation}
The inverse Fourier transform $u_{\delta}^{\varepsilon}\left(\xi,z\right)$ of $\hat{u}_{\delta}^{\varepsilon}\left(\rho,z\right)$ is
\begin{equation}
u_{\delta}^{\varepsilon}\left(\xi,z\right)=\int_{\mathbb{R}^{2}}\left(\hat{g}_{\delta}^{\varepsilon}\left(\rho\right)\cosh\left(\left(d-z\right)\sqrt{\lambda_{\rho,k}}\right)+\frac{1}{\sqrt{\lambda_{\rho,k}}}\int_{z}^{d}\hat{f}_{\delta}^{\varepsilon}\left(\rho,s\right)\sinh\left(\left(z-s\right)\sqrt{\lambda_{\rho,k}}\right)ds\right)e^{2\pi i\left\langle \xi,\rho\right\rangle }d\rho.\label{eq:udelta}
\end{equation}
This function shall  be considered as a regularized solution to
the problem \eqref{eq:helmholtz} for measured data $f_{\delta}$ and $g_{\delta}$ and $h\equiv0$, where the regularization parameter $\varepsilon$
depends on the noise level
$\delta$ and shall be explicitly computed in our main theorem. The
regularized solution $\ensuremath{u^{\varepsilon}\left(\xi,z\right)}$ for the exact data can be defined in the same manner as \eqref{eq:udelta}.

The following lemma proves that under the assumption \textbf{(A3)}, the regularized solution for the exact data is expected to approach the exact solution in $L^{2}$-norm.
\begin{lem}
\label{lem:7}Let $u$ be the unique solution of the problem \eqref{eq:helmholtz}
with the exact data $f$ and $g$ and $h\equiv0$ and let $u^{\varepsilon}$ be the regularized solution for the exact data which is defined in the same manner as \eqref{eq:udelta}. Assume that
\[\int_{\mathbb{R}^2}\left|\hat{g}(\rho)\cosh\left(d\sqrt{\lambda_{\rho,k}}\right)\right|^2d\rho+\int_{\mathbb{R}^2}\left|\frac{1}{\sqrt{\lambda_{\rho,k}}}\int_0^d\left|\hat{f}(\rho,s)\sinh\left(s\sqrt{\lambda_{\rho,k}}\right)\right|ds\right|^2d\rho\le M_0^2,\]
for some given constant $M_0>0$. Then, for $z\in\left(0,d\right]$, we obtain
\[
\left\Vert u\left(\cdot,z\right)-u^{\varepsilon}\left(\cdot,z\right)\right\Vert _{L^{2}\left(\mathbb{R}^{2}\right)}\le M_0\left[e^{-z\sqrt{\frac{1}{\varepsilon}-k^{2}}}\left(1+e^{-2\left(d-z\right)\sqrt{\frac{1}{\varepsilon}-k^{2}}}\right)+e^{-z\sqrt{\frac{1}{\varepsilon}-k^{2}}}\right].
\]
As a consequence, it holds that
\[
\left\Vert u\left(\cdot,z\right)-u^{\varepsilon}\left(\cdot,z\right)\right\Vert _{L^{2}\left(\mathbb{R}^{2}\right)}\to0\quad\mbox{as}\quad\varepsilon\to0.
\]
\end{lem}
\begin{proof}
Since $\hat{u}^\varepsilon$ agrees with $\hat{u}$ when $|\rho|^2\le\frac{1}{\varepsilon}$ and $\hat{u}^\varepsilon=0$ if $|\rho|^2>\frac{1}{\varepsilon}$, we have that
\begin{eqnarray}
\int_{\mathbb{R}^2}\left|\hat{u}\left(\rho,z\right)-\hat{u}^{\varepsilon}\left(\rho,z\right)\right|^{2}d\rho & = & \int_{|\rho|^2>\frac{1}{\varepsilon}}\left|\hat{u}\left(\rho,z\right)-\hat{u}_{\varepsilon}\left(\rho,z\right)\right|\left|\hat{u}\left(\rho,z\right)\right|d\rho\nonumber \\
 & = & \int_{|\rho|^2>\frac{1}{\varepsilon}}\left|\hat{u}\left(\rho,z\right)-\hat{u}_{\varepsilon}\left(\rho,z\right)\right|\left|\hat{g}\left(\rho\right)\cosh\left(\left(d-z\right)\sqrt{\lambda_{\rho,k}}\right)\right.\nonumber \\
 &  & \left.+\frac{1}{\sqrt{\lambda_{\rho,k}}}\int_{z}^{d}\hat{f}\left(\rho,s\right)\sinh\left(\left(z-s\right)\sqrt{\lambda_{\rho,k}}\right)ds\right|d\rho.\label{eq:3.3-1}
\end{eqnarray}

When $A_1$ happens, $|\rho|^2>k^2$ implies $\sqrt{\lambda_{\rho,k}} >0$ and hence, $\cosh(d\sqrt{\lambda_{\rho,k}})\ge 1$. Thanks to this, we can write
\begin{eqnarray*}
\mathcal{I}&:=&\int_{|\rho|^2>\frac{1}{\varepsilon}}\left|\hat{u}\left(\rho,z\right)-\hat{u}^{\varepsilon}\left(\rho,z\right)\right|\left|\hat{g}\left(\rho\right)\cosh\left(\left(d-z\right)\sqrt{\lambda_{\rho,k}}\right)\right|d\rho\\
&=&\int_{|\rho|^2>\frac{1}{\varepsilon}}\left|\hat{u}\left(\rho,z\right)-\hat{u}^{\varepsilon}\left(\rho,z\right)\right|\left|\hat{g}\left(\rho\right)\cosh\left(d\sqrt{\lambda_{\rho,k}}\right)\frac{\cosh\left(\left(d-z\right)\sqrt{\lambda_{\rho,k}}\right)}{\cosh\left(d\sqrt{\lambda_{\rho,k}}\right)}\right|d\rho,
\end{eqnarray*}
Using the H\"{o}lder's inequality, the integral $\mathcal{I}$ can
be bounded from above by
\begin{equation}
\mathcal{I}\le\mathcal{I}_{\varepsilon}\left(z\right)\left\Vert \hat{g}\left(\cdot\right)\cosh(d\sqrt{\lambda_{\rho,k}})\right\Vert _{L^{2}\left(\mathbb{R}^{2}\right)}\left\Vert \hat{u}\left(\cdot,z\right)-\hat{u}^{\varepsilon}\left(\cdot,z\right)\right\Vert _{L^{2}\left(\mathbb{R}^{2}\right)},\label{eq:3.3}
\end{equation}
where the function $\mathcal{I}_{\varepsilon}\left(z\right)$ is defined
and estimated as follows
\begin{eqnarray}
\mathcal{I}_{\varepsilon}\left(z\right) & := & \sup_{|\rho|^2>\frac{1}{\varepsilon}}\left|\frac{\cosh\left(\left(d-z\right)\sqrt{\lambda_{\rho,k}}\right)}{\cosh\left(d\sqrt{\lambda_{\rho,k}}\right)}\right|= \sup_{|\rho|^2>\frac{1}{\varepsilon}}\frac{e^{\left(d-z\right)\sqrt{\lambda_{\rho,k}}}+e^{-\left(d-z\right)\sqrt{\lambda_{\rho,k}}}}{e^{d\sqrt{\lambda_{\rho,k}}}+e^{-d\sqrt{\lambda_{\rho,k}}}}\nonumber \\
 & = & \sup_{|\rho|^2>\frac{1}{\varepsilon}}e^{-z\sqrt{\lambda_{\rho,k}}}\left(\frac{1+e^{-2\left(d-z\right)\sqrt{\lambda_{\rho,k}}}}{1+e^{-2d\sqrt{\lambda_{\rho,k}}}}\right)\le  \sup_{|\rho|^2>\frac{1}{\varepsilon}}e^{-z\sqrt{\lambda_{\rho,k}}}\left(1+e^{-2\left(d-z\right)\sqrt{\lambda_{\rho,k}}}\right).\label{eq:3.4}
\end{eqnarray}
Due to the fact that the function $s\mapsto e^{-z\sqrt{s-k^{2}}}\left(1+e^{-2\left(d-z\right)\sqrt{s-k^{2}}}\right)$
is decreasing, the estimates
\eqref{eq:3.3} and \eqref{eq:3.4} imply
\begin{equation}
\mathcal{I}\le M_0e^{-z\sqrt{\frac{1}{\varepsilon}-k^{2}}}\left(1+e^{-2\left(d-z\right)\sqrt{\frac{1}{\varepsilon}-k^{2}}}\right)\left\Vert \hat{u}\left(\cdot,z\right)-\hat{u}^{\varepsilon}\left(\cdot,z\right)\right\Vert _{L^{2}\left(\mathbb{R}^{2}\right)}.\label{eq:I}
\end{equation}

In the same vein, we have $\ensuremath{\sinh\left(s\sqrt{\lambda_{\rho,k}}\right)\ne0}>0$ for all $s\in (0,d)$. It holds that
\begin{eqnarray}
\mathcal{J} &:=&\int_{|\rho|^2>\frac{1}{\varepsilon}}\left|\hat{u}\left(\rho,z\right)-\hat{u}^{\varepsilon}\left(\rho,z\right)\right|\left|\frac{1}{\sqrt{\lambda_{\rho,k}}}\int_{z}^{d}\hat{f}\left(\rho,s\right)\sinh\left(\left(s-z\right)\sqrt{\lambda_{\rho,k}}\right)ds\right|d\rho\nonumber\\
& \le & \int_{|\rho|^2>\frac{1}{\varepsilon}}\left|\hat{u}\left(\rho,z\right)-\hat{u}^{\varepsilon}\left(\rho,z\right)\right|\frac{1}{\sqrt{\lambda_{\rho,k}}}\int_{z}^{d}\left|\hat{f}\left(\rho,s\right)\sinh\left(s\sqrt{\lambda_{\rho,k}}\right)\right|\left|\frac{\sinh\left(\left(s-z\right)\sqrt{\lambda_{\rho,k}}\right)}{\sinh\left(s\sqrt{\lambda_{\rho,k}}\right)}\right|dsd\rho\nonumber \\
 & \le & M_0\mathcal{J}_{\varepsilon}\left(s\right)\left\Vert \hat{u}\left(\cdot,z\right)-\hat{u}^{\varepsilon}\left(\cdot,z\right)\right\Vert _{L^{2}\left(\mathbb{R}^{2}\right)},\label{eq:3.7}
\end{eqnarray}
where $\mathcal{J}_{\varepsilon}\left(s\right),z\le s\le d$, is
defined and estimated as follows:
\begin{eqnarray}
\mathcal{J}_{\varepsilon}\left(s\right) & := & \sup_{|\rho|^2>\frac{1}{\varepsilon}}\left|\frac{\sinh\left(\left(s-z\right)\sqrt{\lambda_{\rho,k}}\right)}{\sinh\left(s\sqrt{\lambda_{\rho,k}}\right)}\right| =  \sup_{|\rho|^2>\frac{1}{\varepsilon}}\left|\frac{e^{\left(s-z\right)\sqrt{\lambda_{\rho,k}}}-e^{-\left(s-z\right)\sqrt{\lambda_{\rho,k}}}}{e^{s\sqrt{\lambda_{\rho,k}}}-e^{-s\sqrt{\lambda_{\rho,k}}}}\right|\nonumber \\
 & = & \sup_{|\rho|^2>\frac{1}{\varepsilon}}e^{-z\sqrt{\lambda_{\rho,k}}}\left|\frac{1-e^{-2\left(s-z\right)\sqrt{\lambda_{\rho,k}}}}{1-e^{-2s\sqrt{\lambda_{\rho,k}}}}\right| \le  e^{-z\sqrt{\frac{1}{\varepsilon}-k^{2}}}.\label{eq:J}
\end{eqnarray}

Combining \eqref{eq:I}, \eqref{eq:3.7} and \eqref{eq:J} together
with \eqref{eq:3.3-1}, and employing the Parseval's identity, we obtain
\[
\left\Vert u\left(\cdot,z\right)-u^{\varepsilon}\left(\cdot,z\right)\right\Vert _{L^{2}\left(\mathbb{R}^{2}\right)}^{2}\le M_0\left[e^{-z\sqrt{\frac{1}{\varepsilon}-k^{2}}}\left(1+e^{-2\left(d-z\right)\sqrt{\frac{1}{\varepsilon}-k^{2}}}\right)+e^{-z\sqrt{\frac{1}{\varepsilon}-k^{2}}}\right]\left\Vert u\left(\cdot,z\right)-u^{\varepsilon}\left(\cdot,z\right)\right\Vert _{L^{2}\left(\mathbb{R}^{2}\right)}.
\]
This completes the proof of the lemma.\end{proof}

The regularized solution for the exact data $u^{\varepsilon}$, on the other hand, can be approximated by the regularized solution for the measured data $u^{\varepsilon}_\delta$. This result is proved in the following lemma.
\begin{lem}
\label{lem:9}Let $u^{\varepsilon}$ and $u_{\delta}^{\varepsilon}$
be two regularized solutions defined in \eqref{eq:udelta} with respect to the data $\ensuremath{\left(f,g\right)}$ and $\ensuremath{\left(f_{\delta},g_{\delta}\right)}$
respectively, where the measured data $f_{\delta}$ and $g_{\delta}$ satisfy \textbf{(A1)}.
Then for $z\in\left[0,d\right]$, there exists a function $M_{1}\left(z\right)>0$
depending only on $k,\varepsilon$ and $d$ such that
\[
\left\Vert u^{\varepsilon}\left(\cdot,z\right)-u_{\delta}^{\varepsilon}\left(\cdot,z\right)\right\Vert _{L^{2}\left(\mathbb{R}^{2}\right)}\le M_{1}\left(z\right)\delta.
\]
\end{lem}
\begin{proof}
It follows from Parseval's identity, the Cauchy-Schwarz inequality and H\"{o}lder's inequality that
\begin{eqnarray}
\left\Vert u^{\varepsilon}\left(\cdot,z\right)-u_{\delta}^{\varepsilon}\left(\cdot,z\right)\right\Vert _{L^{2}\left(\mathbb{R}^{2}\right)}^{2} &=& \left\Vert \hat{u}^{\varepsilon}\left(\cdot,z\right)-\hat{u}_{\delta}^{\varepsilon}\left(\cdot,z\right)\right\Vert _{L^{2}\left(\mathbb{R}^{2}\right)}^{2}\nonumber\\
&=&\int_{|\rho|^2\le\frac{1}{\varepsilon}}\left|\left(\hat{g}-\hat{g}_\delta\right)\cosh\left((d-z)\sqrt{\lambda_{\rho,k}}\right)+\frac{1}{\sqrt{\lambda_{\rho,k}}}\int_z^d\left(\hat{f}-\hat{f}_\delta\right)\sinh\left((z-s)\sqrt{\lambda_{\rho,k}}\right)ds\right|^2d\rho\nonumber\\
& \le & 2\int_{|\rho|^2\le\frac{1}{\varepsilon}}\left[\left|\hat{g}\left(\rho\right)-\hat{g}_{\delta}\left(\rho\right)\right|^{2}\cosh^{2}\left(\left(d-z\right)\sqrt{\lambda_{\rho,k}}\right)+\right.\nonumber \\
 &  & \left.+\frac{1}{\lambda_{\rho,k}}\left(d-z\right)\int_{z}^{d}\left|\hat{f}\left(\rho,s\right)-\hat{f}_{\delta}\left(\rho,s\right)\right|^{2}\sinh^{2}\left(\left(s-z\right)\sqrt{\lambda_{\rho,k}}\right)ds\right]d\rho\nonumber \\
 & \le & 2\delta^{2}\sup_{|\rho|^2\le\frac{1}{\varepsilon}}\left[\cosh^{2}\left(\left(d-z\right)\sqrt{\lambda_{\rho,k}}\right)+\left(d-z\right)\frac{1}{\lambda_{\rho,k}}\int_{z}^{d}\sinh^{2}\left(\left(s-z\right)\sqrt{\lambda_{\rho,k}}\right)ds\right].\label{eq:3.9}
\end{eqnarray}
Since $\ensuremath{\rho\in A_{1}}$,
the function $r\mapsto \ensuremath{\cosh\left(\left(d-z\right)\sqrt{r-k^2}\right)}$
is increasing. On the other hand,
\begin{eqnarray*}
\int_{z}^{d}\sinh^{2}\left(\left(s-z\right)\sqrt{\lambda_{\rho,k}}\right)ds & = & \frac{z-d}{2}+\frac{\sinh\left(2\sqrt{\lambda_{\rho,k}}\left(d-z\right)\right)}{4\sqrt{\lambda_{\rho,k}}}\\
 & = & \frac{\sinh\left(2\sqrt{\lambda_{\rho,k}}\left(d-z\right)\right)-2\sqrt{\lambda_{\rho,k}}\left(d-z\right)}{4\sqrt{\lambda_{\rho,k}}},
\end{eqnarray*}
and the function $r \mapsto \ensuremath{{\displaystyle \frac{\sinh\left(2\sqrt{r-k^{2}}\left(d-z\right)\right)-2\sqrt{r-k^{2}}\left(d-z\right)}{4\left(\sqrt{r-k^{2}}\right)^{3}}}}$
is increasing. These imply that the supremum on the right-hand
side of \eqref{eq:3.9} is attained when $|\rho|^2=\frac{1}{\varepsilon}$. This supremum is given by
\begin{equation}\label{eq:240216}
\cosh^{2}\left(\left(d-z\right)\sqrt{\frac{1}{\varepsilon}-k^{2}}\right)+{\displaystyle \left(d-z\right)\frac{\sinh\left(2\sqrt{\frac{1}{\varepsilon}-k^{2}}\left(d-z\right)\right)-2\sqrt{\frac{1}{\varepsilon}-k^{2}}\left(d-z\right)}{4\left(\sqrt{\frac{1}{\varepsilon}-k^{2}}\right)^{3}}}.
\end{equation}
Combining \eqref{eq:3.9}-\eqref{eq:240216}, we conclude that
there exists a function $M_{1}\left(z\right)>0$ depending on $k,\varepsilon$
and $d$ such that 
\[
\left\Vert u^{\varepsilon}\left(\cdot,z\right)-u_{\delta}^{\varepsilon}\left(\cdot,z\right)\right\Vert _{L^{2}\left(\mathbb{R}^{2}\right)}\le M_{1}\left(z\right)\delta,\quad z\in\left[0,d\right].
\] 
Interestingly, this function can be chosen suitably and it is particularly formulated by
\[
M_{1}^{2}\left(z\right)=2\cosh^{2}\left(\left(d-z\right)\sqrt{\frac{1}{\varepsilon}-k^{2}}\right)+\left(d-z\right)\frac{\sinh\left(2\sqrt{\frac{1}{\varepsilon}-k^{2}}\left(d-z\right)\right)-2\sqrt{\frac{1}{\varepsilon}-k^{2}}\left(d-z\right)}{2\left(\sqrt{\frac{1}{\varepsilon}-k^{2}}\right)^{3}}.
\]
This ends the proof of the lemma.\end{proof}
\begin{rem}
The idea leading to Lemma \ref{lem:7} and Lemma \ref{lem:9} is that
one needs to estimate the error in $L^{2}$-norm between the exact
solution of \eqref{eq:helmholtz} where $h\equiv0$ and the regularized
solution proposed in \eqref{eq:udelta}. It is worth noting that
  when $\delta$ approaches zero,
the way we choose $\varepsilon$ must guarantee that $\frac{1}{\varepsilon}$ spreads to infinity
and $\frac{1}{\varepsilon}>k^{2}$ holds. Therefore, the following
theorem gives us exactly what we need.\end{rem}
\begin{thm}
\label{thm:11}Let $u$ be the exact solution of the problem \eqref{eq:helmholtz}
with $h\equiv0$. Let $u_{\delta}^{\varepsilon}$ be the regularized
solution given by \eqref{eq:udelta} associated with the measured data $f_{\delta}$
and $g_{\delta}$. We assume that the noise level $\delta<M_{0}$ where $M_0$ is the constant in Lemma \ref{lem:7}, and that the measured data $f_\delta$ and $g_\delta$ satisfy \textbf{(A1)}. If we put $\kappa_{\varepsilon}:=\sqrt{\frac{1}{\varepsilon}-k^{2}}$
and $\varepsilon:=\varepsilon\left(\delta\right)$ such that
\[
\kappa_{\varepsilon\left(\delta\right)}=-\frac{1}{d}\ln\frac{\delta}{M_{0}},
\]
then for every $z\in\left(0,d\right]$, we obtain the estimate
\[
\left\Vert u\left(\cdot,z\right)-u_{\delta}^{\varepsilon}\left(\cdot,z\right)\right\Vert _{L^{2}\left(\mathbb{R}^{2}\right)}\le\left(2\sqrt{2\delta^{2\left(1-\frac{z}{d}\right)}+M_{0}^{\frac{2\left(d-z\right)}{d}}\left[1+\frac{d^{3}\left(d-z\right)}{4\ln^{3}\left(\frac{M_{0}}{\delta}\right)}\right]}+M_{0}^{\frac{d-z}{d}}\right)\delta^{\frac{z}{d}}.
\]

As a consequence, for each $z\in\left(0,d\right]$, we have
\[
\left\Vert u\left(\cdot,z\right)-u_{\delta}^{\varepsilon}\left(\cdot,z\right)\right\Vert _{L^{2}\left(\mathbb{R}^{2}\right)}\to0 \quad\mbox{as}\quad\delta\to0.
\]
\end{thm}
\begin{proof}
Thanks to Lemma \ref{lem:7} and Lemma \ref{lem:9}, it follows from the triangle
inequality that
\begin{eqnarray}
\left\Vert u\left(\cdot,z\right)-u_{\delta}^{\varepsilon}\left(\cdot,z\right)\right\Vert _{L^{2}\left(\mathbb{R}^{2}\right)} & \le & \left\Vert u\left(\cdot,z\right)-u^{\varepsilon}\left(\cdot,z\right)\right\Vert _{L^{2}\left(\mathbb{R}^{2}\right)}+\left\Vert u^{\varepsilon}\left(\cdot,z\right)-u_{\delta}^{\varepsilon}\left(\cdot,z\right)\right\Vert _{L^{2}\left(\mathbb{R}^{2}\right)}\nonumber \\
 & \le & \left[e^{-z\kappa_{\varepsilon}}\left(1+e^{-2\left(d-z\right)\kappa_{\varepsilon}}\right)+e^{-z\kappa_{\varepsilon}}\right]M_{0}+M_{1}\delta\nonumber \\
 & \le & \left[e^{-d\kappa_{\varepsilon}}\left(e^{\left(d-z\right)\kappa_{\varepsilon}}+e^{-\left(d-z\right)\kappa_{\varepsilon}}\right)+e^{-z\kappa_{\varepsilon}}\right]M_{0}+M_{1}\delta\nonumber \\
 & \le & \left[2e^{-d\kappa_{\varepsilon}}\cosh\left(\left(d-z\right)\kappa_{\varepsilon}\right)+e^{-z\kappa_{\varepsilon}}\right]M_{0}+M_{1}\delta\nonumber \\
 & \le & \left[e^{-d\kappa_{\varepsilon}}M_{1}+e^{-z\kappa_{\varepsilon}}\right]M_{0}+M_{1}\delta\nonumber \\
 & \le & M_{1}\left(M_{0}e^{-d\kappa_{\varepsilon}}+\delta\right)+e^{-z\kappa_{\varepsilon}}M_{0},\label{eq:3.13}
\end{eqnarray}
where $M_{1}$ is known in Lemma \ref{lem:9} as the function with
respect to $z$, which reads
\begin{equation}
M_{1}^{2}\left(z\right):=2\cosh^{2}\left(\left(d-z\right)\kappa_{\varepsilon}\right)+{\displaystyle \left(d-z\right)\frac{\sinh\left(2\kappa_{\varepsilon}\left(d-z\right)\right)-2\kappa_{\varepsilon}\left(d-z\right)}{2\left(\kappa_{\varepsilon}\right)^{3}}}.\label{eq:3.14}
\end{equation}

For $\varepsilon=\varepsilon\left(\delta\right)$, we have $e^{-d\kappa_{\varepsilon}}=\dfrac{\delta}{M_{0}}$
and
\[
e^{\left(d-z\right)\kappa_{\varepsilon}}=\left(\frac{\delta}{M_{0}}\right)^{\frac{z-d}{d}}.
\]
Substituting this into \eqref{eq:3.13}, we thus have
\begin{equation}
\left\Vert u\left(\cdot,z\right)-u_{\delta}^{\varepsilon}\left(\cdot,z\right)\right\Vert _{L^{2}\left(\mathbb{R}^{2}\right)}\le2M_{1}\delta+\left(\frac{\delta}{M_{0}}\right)^{\frac{z}{d}}M_{0}.\label{eq:3.15}
\end{equation}

As mentioned above, we are interested in $\delta\to0$ which implies $\kappa_{\varepsilon}\to\infty$,
so we will now consider the following elementary results:
\begin{eqnarray}
\cosh^{2}\left(\left(d-z\right)\kappa_{\varepsilon}\right) & = & \left(\frac{e^{\left(d-z\right)\kappa_{\varepsilon}}+e^{-\left(d-z\right)\kappa_{\varepsilon}}}{2}\right)^{2} =  \frac{1+e^{2\left(d-z\right)\kappa_{\varepsilon}}+e^{-2\left(d-z\right)\kappa_{\varepsilon}}}{2}\nonumber\\
 & = & \frac{1}{2}\left[1+\left(\frac{\delta}{M_{0}}\right)^{\frac{2\left(z-d\right)}{d}}+\left(\frac{\delta}{M_{0}}\right)^{\frac{2\left(d-z\right)}{d}}\right] \le  \frac{1}{2}\left(2+M_{0}^{\frac{2\left(d-z\right)}{d}}\delta^{\frac{2\left(z-d\right)}{d}}\right),\label{coshhhh}
\end{eqnarray}
\begin{equation}
\frac{\sinh\left(2\kappa_{\varepsilon}\left(d-z\right)\right)-2\kappa_{\varepsilon}\left(d-z\right)}{2\left(\kappa_{\varepsilon}\right)^{3}}\le\frac{e^{2\kappa_{\varepsilon}\left(d-z\right)}-e^{2\kappa_{\varepsilon}\left(z-d\right)}}{4\left(\kappa_{\varepsilon}\right)^{3}}\le d^{3}\frac{M_{0}^{\frac{2\left(d-z\right)}{d}}\delta^{\frac{2\left(z-d\right)}{d}}}{4\ln^{3}\left(\frac{M_{0}}{\delta}\right)}.\label{sinhhhh}
\end{equation}
Combining these inequalities, it follows from \eqref{eq:3.14} that
\begin{eqnarray}
M_{1} \le \sqrt{2+M_{0}^{\frac{2\left(d-z\right)}{d}}\delta^{\frac{2\left(z-d\right)}{d}}\left[1+\frac{d^{3}\left(d-z\right)}{4\ln^{3}\left(\frac{M_{0}}{\delta}\right)}\right]}.\label{eq:3.16}
\end{eqnarray}
Thanks to \eqref{eq:3.15} and \eqref{eq:3.16}, we arrive at
\begin{eqnarray*}
\left\Vert u\left(\cdot,z\right)-u_{\delta}^{\varepsilon}\left(\cdot,z\right)\right\Vert _{L^{2}\left(\mathbb{R}^{2}\right)} & \le & 2\sqrt{2+M_{0}^{\frac{2\left(d-z\right)}{d}}\delta^{\frac{2\left(z-d\right)}{d}}\left[1+\frac{d^{3}\left(d-z\right)}{4\ln^{3}\left(\frac{M_{0}}{\delta}\right)}\right]}\delta+\delta^{\frac{z}{d}}M_{0}^{\frac{d-z}{d}} \nonumber\\
 & \le & \left(2\sqrt{2\delta^{2\left(1-\frac{z}{d}\right)}+M_{0}^{\frac{2\left(d-z\right)}{d}}\left[1+\frac{d^{3}\left(d-z\right)}{4\ln^{3}\left(\frac{M_{0}}{\delta}\right)}\right]}+M_{0}^{\frac{d-z}{d}}\right)\delta^{\frac{z}{d}},
\end{eqnarray*}
which completes the error estimate as well as the statement of strong convergence in $L^2$-norm.
\end{proof}
\begin{rem}
In Theorem \ref{thm:11}, we give  a convergent approximation
of $u\left(\xi,z\right)$ . Moreover,
for $z\in\left(0,d\right]$, the error estimate is of the order $\mathcal{O}\left(\delta^{\frac{z}{d}}\right)$.
We also notice that for $z=d$, it coincides with the order given by
\textbf{(A1)}, i.e. the noise level attached adheres the exact data. Thus,
the result obtained here is reasonable.

In addition, the assumption on $\kappa_{\varepsilon}$ gives us an explicit
formula for the regularization parameter $\varepsilon$, which reads
\begin{equation}
\varepsilon\left(\delta\right)=\left(k^{2}+\frac{1}{d^{2}}\ln^{2}\left(\frac{\delta}{M_{0}}\right)\right)^{-1}.\label{eq:pararegu}
\end{equation}
Consequently, \eqref{eq:pararegu} tells us how fast the bounded disk $\Theta_{\varepsilon}$ enlarges, provided by
\begin{equation}
|\Theta_{\varepsilon}|\le k^2 + \frac{1}{d^{2}}\ln^{2}\left(\frac{\delta}{M_{0}}\right),
\end{equation}
where $|\Theta_{\varepsilon}|$ denotes the Lebesgue measure (i.e. the area in this case) of $\Theta_{\varepsilon}$.
\end{rem}

\begin{rem}
The \emph{a priori} condition in Lemma \ref{lem:7} extends greatly the restriction in the work in \cite{RR06}. Indeed, this fact may be readily ascertained when the absence of $f$ happens. On the other side, if $fg\le 0$ in $\Omega$ then this condition reduces completely to the energy $\left\Vert u\left(\cdot,0\right)\right\Vert _{L^{2}\left(\mathbb{R}^{2}\right)}$. In practice, we usually do not know exactly the values at the original
point, says $u\left(x,y,0\right)$, so computing the parameter $\varepsilon$
given by \eqref{eq:pararegu} related to this energy
is more or less impossible. Therefore,
one may discuss another choice in the following theorem.\end{rem}

\begin{thm}\label{thm:13}
In Theorem \ref{thm:11}, for every noise level $\delta>0$ we put $\kappa_{\varepsilon}:=\sqrt{\frac{1}{\varepsilon}-k^2}$ and $\varepsilon:=\varepsilon(\delta)$ such that the function $\delta\mapsto e^{d\kappa_{\varepsilon}}\delta$ is non-increasing and that
\begin{equation}
\lim_{\delta\to 0}e^{d\kappa_{\varepsilon}}\delta\le P<\infty,\label{newcondddd}
\end{equation}
for some positive constant $P$, then the following estimate holds:
\[
\left\Vert u\left(\cdot,z\right)-u_{\delta}^{\varepsilon}\left(\cdot,z\right)\right\Vert _{L^{2}\left(\mathbb{R}^{2}\right)}\le\left(M_{0}+P\right)\sqrt{2+\frac{d-z}{4\left(\kappa_{\varepsilon}\right)^{3}}}e^{-z\kappa_{\varepsilon}}.
\]
\end{thm}
\begin{proof}
We start the proof by slightly modifying \eqref{coshhhh} and \eqref{sinhhhh}, as follows:
\[
\cosh^{2}\left((d-z)\kappa_{\varepsilon}\right)\le e^{2(d-z)\kappa_{\varepsilon}},
\] 
\[
\frac{\sinh\left(2\kappa_{\varepsilon}(d-z)\right)-2\kappa_{\varepsilon}(d-z)}{2\left(\kappa_{\varepsilon}\right)^3}
\le \frac{e^{2\kappa_{\varepsilon}(d-z)}}{4\left(\kappa_{\varepsilon}\right)^3}.
\]
Combining the above inequalities and recalling \eqref{eq:3.14}, one deduces that
\begin{equation}
M^{2}_{1}\le 2e^{2(d-z)\kappa_{\varepsilon}} + (d-z)\frac{e^{2\kappa_{\varepsilon}(d-z)}}{4\left(\kappa_{\varepsilon}\right)^3},\label{M1est}
\end{equation}
which naturally implies that
\begin{equation}
M_{1}e^{-d\kappa_{\varepsilon}}\le \sqrt{2+\frac{d-z}{4\left(\kappa_{\varepsilon}\right)^3}}e^{-z\kappa_{\varepsilon}}.\label{M1estnew}
\end{equation}
Hence, from \eqref{eq:3.13}, \eqref{M1est} and \eqref{M1estnew} we arrive at
\begin{eqnarray*}
\left\Vert u\left(\cdot,z\right)-u_{\delta}^{\varepsilon}\left(\cdot,z\right)\right\Vert _{L^{2}\left(\mathbb{R}^{2}\right)} & \le & M_{0}M_{1}e^{-d\kappa_{\varepsilon}}+M_{0}e^{-z\kappa_{\varepsilon}}+\delta M_{1}\\
 & \le & \left(M_{0}+e^{d\kappa_{\varepsilon}}\delta\right)\sqrt{2+\frac{d-z}{4\left(\kappa_{\varepsilon}\right)^{3}}}e^{-z\kappa_{\varepsilon}}.
\end{eqnarray*}
This completes the proof of the theorem.
\end{proof}

\begin{rem}
As we readily expected, by choosing $\kappa_{\varepsilon}$ large enough, a good approximation is obtained. In fact, a very simple example to consider is the choice
\[
\kappa_{\varepsilon}=\frac{1}{d}\ln\left(\frac{1}{\delta}\right).
\]
At this point we see, it gives us $P=1$ and the error estimate in Theorem \ref{thm:13} is of the order $\mathcal{O}\left(\delta^{\frac{z}{d}}\right)$ regardless of $M_0$ or the energy $\left\Vert u\left(\cdot,0\right)\right\Vert _{L^{2}\left(\mathbb{R}^{2}\right)}$ in some certain cases. This approach is incentive for us to take into consideration the nonlinear part.  
\end{rem}

\section{An extension to a nonlinear forcing term}\label{Sec:extension}

Our motivation for this section originates from the Schr\"odinger equation for the wave function $\Psi\left(t,x\right)$
of a particle in quantum mechanics, which reads
\begin{equation}
i\hbar\Psi_{t}=-\frac{\hbar^{2}}{2m}\Delta\Psi+V\left(x\right)\Psi,\label{eq:schrodinger}
\end{equation}
where $\hbar$ is Planck\textquoteright s constant, $m$ is the particle
mass and $V\left(x\right)$ represents the potential. 

Assume that $\Psi$ has a fixed oscillation in time, i.e.,
\[
\Psi\left(t,x\right)=\psi\left(x\right)e^{-ikt},
\]
then we can substitute this into \eqref{eq:schrodinger} and divide both
sides by $e^{-ikt}$ to get
\[
\frac{\hbar^{2}}{2m}\Delta\psi+E\psi=V\left(x\right)\psi,
\]
where $E=\hbar k$ is called the total energy in the quantum setting,
and herein  $\psi$  is viewed as a stationary or orbital state with
$\left|\psi\right|^{2}$, the probability distribution
of the spatial location for a particle at a fixed energy $E$.

Therefore, we shall consider the following Helmholtz equation associated
with nonlinear forcing term:
\begin{equation}
\begin{cases}
\Delta u+k^{2}u=f\left(u\right), & \mbox{in}\;\Omega,\\
u\left(\xi,d\right)=g\left(\xi\right), & \xi\in\mathbb{R}^{2},\\
\partial_{z}u\left(\xi,d\right)=h\left(\xi\right), & \xi\in\mathbb{R}^{2},\\
u\left(\cdot,z\right)\in L^{2}\left(\mathbb{R}^{2}\right), & z\in\left[0,d\right],
\end{cases}\label{eq:nonlinear1}
\end{equation}
where $f$ is a globally Lipschitz forcing term:
\[
\left\Vert f\left(\cdot,z,u\right)-f\left(\cdot,z,v\right)\right\Vert _{L^{2}\left(\mathbb{R}^{2}\right)}\le L_{f}\left\Vert u\left(\cdot,z\right)-v\left(\cdot,z\right)\right\Vert _{L^{2}\left(\mathbb{R}^{2}\right)},\quad z\in\left[0,d\right],
\]
for some Lipschitz constant $L_{f}>0$. It is remarkable that the solution to \eqref{eq:nonlinear1} can be computed
using the nonlinear spectral theory, but we will not go further in
this direction since such techniques can be found in
\cite{TTKT15}. In this section, we shall focus on the regularization for the so-called
mild solution of \eqref{eq:nonlinear1}.

Denote by $F\left(u\right):=f\left(u\right)-k^{2}u$, we see that $F$ is still globally Lipschitz thanks to the triangle inequality,
where the Lipschitz constant is given by $\ell_{f}:=L_{f}+k^{2}$. Solving  \eqref{eq:nonlinear1} with the first equation replaced by
\[\Delta u = F(u),\]
we get
\begin{equation}
\hat{u}\left(\xi,z\right)  =  \hat{g}\left(\rho\right)\cosh\left(\left(d-z\right)\sqrt{\lambda_{\rho,0}}\right)+\frac{\hat{h}\left(\rho\right)\sinh\left(\left(d-z\right)\sqrt{\lambda_{\rho,0}}\right)}{\sqrt{\lambda_{\rho,0}}} +\int_{z}^{d}\frac{\sinh\left(\left(s-z\right)\sqrt{\lambda_{\rho,0}}\right)}{\sqrt{\lambda_{\rho,0}}}\hat{F}\left(\rho,s,u\right)ds.\label{uhatnonlinear}
\end{equation}
  Hereby, we say that the mild solution of \eqref{eq:nonlinear1} is a function $u\in C\left(\left[0,d\right];L^{2}\left(\mathbb{R}^{2}\right)\right)$ satisfying
\begin{eqnarray}
u\left(\xi,z\right) & = & \int_{\mathbb{R}^{2}}\left(\hat{g}\left(\rho\right)\cosh\left(\left(d-z\right)\sqrt{\lambda_{\rho,0}}\right)+\frac{\hat{h}\left(\rho\right)\sinh\left(\left(d-z\right)\sqrt{\lambda_{\rho,0}}\right)}{\sqrt{\lambda_{\rho,0}}}\right.\nonumber\\
 &  & \left.+\int_{z}^{d}\frac{\sinh\left(\left(s-z\right)\sqrt{\lambda_{\rho,0}}\right)}{\sqrt{\lambda_{\rho,0}}}\hat{F}\left(\rho,s,u\right)ds\right)e^{2\pi i\left\langle \xi,\rho\right\rangle }d\rho.\label{eq:unon}
\end{eqnarray}

Similarly to \eqref{eq:udelta}, our regularized solution for \eqref{eq:unon}
corresponding to measured data $g_{\delta},h_{\delta}$ in \textbf{(A1)}
can be written as
\begin{eqnarray}
u_{\delta}^{\varepsilon}\left(\xi,z\right) & = & \int_{\rho\in\Theta_{\varepsilon}}\left(\hat{g}_\delta\left(\rho\right)\cosh\left(\left(d-z\right)\sqrt{\lambda_{\rho,0}}\right)+\frac{\hat{h}_\delta\left(\rho\right)\sinh\left(\left(d-z\right)\sqrt{\lambda_{\rho,0}}\right)}{\sqrt{\lambda_{\rho,0}}}\right.\nonumber\\
 &  & \left.+\int_{z}^{d}\frac{\sinh\left(\left(s-z\right)\sqrt{\lambda_{\rho,0}}\right)}{\sqrt{\lambda_{\rho,0}}}\hat{F}\left(\rho,s,u_{\delta}^{\varepsilon} \right)ds\right)e^{2\pi i\left\langle \xi,\rho\right\rangle }d\rho.\label{eq:uregunon}
\end{eqnarray}

Aside from the linear section, the representation \eqref{eq:uregunon}
does not make certain of the existence and uniqueness of the regularized
solution. On the other hand, it has not escaped our notice that \eqref{eq:uregunon}
performs exactly the integral equation $w\left(\xi,z\right)=\mathcal{G}\left(w\right)\left(\xi,z\right)$
where the operator $\mathcal{G}$ mapping from $C\left(\left[0,d\right];L^{2}\left(\mathbb{R}^{2}\right)\right)$
into itself is defined by
\begin{eqnarray}
\mathcal{G}\left(w\right)\left(\xi,z\right) & := & \int_{\rho\in\Theta_{\varepsilon}}\left(\hat{g}_\delta\left(\rho\right)\cosh\left(\left(d-z\right)\sqrt{\lambda_{\rho,0}}\right)+\frac{\hat{h}_\delta\left(\rho\right)\sinh\left(\left(d-z\right)\sqrt{\lambda_{\rho,0}}\right)}{\sqrt{\lambda_{\rho,0}}}\right.\nonumber\\
 &  & \left.+\int_{z}^{d}\frac{\sinh\left(\left(s-z\right)\sqrt{\lambda_{\rho,0}}\right)}{\sqrt{\lambda_{\rho,0}}}\hat{F}\left(\rho,s,w \right)ds\right)e^{2\pi i\left\langle \xi,\rho\right\rangle }d\rho.\label{eq:mathG}
\end{eqnarray}

Applying the  Banach fixed-point theorem, we prove the existence
and uniqueness of solution to \eqref{eq:uregunon} by the following lemma. Afterwards, we shall
show the convergence result where the proof also provides the stability
of the regularized solution.
\begin{lem}
The integral equation $w\left(\xi,z\right)=\mathcal{G}\left(w\right)\left(\xi,z\right)$
where $\mathcal{G}$ is given by \eqref{eq:mathG} has a unique solution
$w\in C\left(\left[0,d\right];L^{2}\left(\mathbb{R}^{2}\right)\right)$
for each $\varepsilon>0$.\end{lem}
\begin{proof}
We shall prove by induction that for all $m\in\mathbb{N}$,
\begin{equation}
\left\Vert \mathcal{G}^{m}\left(w_{1}\right)\left(\cdot,z\right)-\mathcal{G}^{m}\left(w_{2}\right)\left(\cdot,z\right)\right\Vert _{L^{2}\left(\mathbb{R}^{2}\right)}\le\sqrt{\left(e^{\frac{2d}{\varepsilon}}d^{2}\ell_{f}^{2}\right)^{m}\frac{\left(d-z\right)^{m}}{m!}}\left\Vert w_{1}-w_{2}\right\Vert _{C\left(\left[0,d\right];L^{2}\left(\mathbb{R}^{2}\right)\right)},\label{eq:3.26}
\end{equation}
for every $\varepsilon>0$ and $w_{1}\left(\cdot,z\right),w_{2}\left(\cdot,z\right)\in C\left(\left[0,d\right];L^{2}\left(\mathbb{R}^{2}\right)\right)$
, $z\in\left[0,d\right]$.

First, for $m=1$,  the  Parseval's identity and the inequality
$\dfrac{\sinh\left(ax\right)}{x}\le ae^{ax}$ for $x>0,a\ge0$ yield
\begin{eqnarray*}
\left\Vert \mathcal{G}\left(w_{1}\right)\left(\cdot,z\right)-\mathcal{G}\left(w_{2}\right)\left(\cdot,z\right)\right\Vert _{L^{2}\left(\mathbb{R}^{2}\right)}^{2} & \le & \left\Vert \int_{z}^{d}e^{\frac{s-z}{\sqrt{\varepsilon}}}\left|\hat{F}\left(\cdot,s,w_{1}\right)-\hat{F}\left(\cdot,s,w_{2}\right)\right|ds\right\Vert _{L^{2}\left(\mathbb{R}^{2}\right)}^{2}\\
 & \le & e^{\frac{2d}{\sqrt{\varepsilon}}}\left(d-z\right)^{2}\int_{z}^{d}\left\Vert F\left(\cdot,s,w_{1}\right)-F\left(\cdot,s,w_{2}\right)\right\Vert _{L^{2}\left(\mathbb{R}^{2}\right)}^{2}ds\\
 & \le & e^{\frac{2d}{\sqrt{\varepsilon}}}\left(d-z\right)^{2}\ell_{f}^{2}\int_{z}^{d}\left\Vert w_{1}\left(\cdot,s\right)-w_{2}\left(\cdot,s\right)\right\Vert _{L^{2}\left(\mathbb{R}^{2}\right)}^{2}ds\\
 & \le & e^{\frac{2d}{\sqrt{\varepsilon}}}d^{2}\ell_{f}^{2}\left(d-z\right)\left\Vert w_{1}-w_{2}\right\Vert _{C\left(\left[0,d\right];L^{2}\left(\mathbb{R}^{2}\right)\right)}^{2}
\end{eqnarray*}
where we have used the fact that $\rho\in\Theta_{\varepsilon}$ ,
H\"{o}lder's inequality and the Lipschitz property of $F$.
Thus, \eqref{eq:3.26} holds for $m=1$.

Suppose that \eqref{eq:3.26} holds up to $m=n$, we shall prove
that it also holds for $m=n+1$. Indeed, we have in a similar manner
that
\begin{eqnarray*}
\left\Vert \mathcal{G}^{n+1}\left(w_{1}\right)\left(\cdot,z\right)-\mathcal{G}^{n+1}\left(w_{2}\right)\left(\cdot,z\right)\right\Vert _{L^{2}\left(\mathbb{R}^{2}\right)}^{2} & \le & e^{\frac{2d}{\sqrt{\varepsilon}}}\left(d-z\right)^{2}\ell_{f}^{2}\int_{z}^{d}\left\Vert \mathcal{G}^{n}\left(w_{1}\right)\left(\cdot,s\right)-\mathcal{G}^{n}\left(w_{2}\right)\left(\cdot,s\right)\right\Vert _{L^{2}\left(\mathbb{R}^{2}\right)}^{2}ds\\
 & \le & e^{\frac{2d}{\sqrt{\varepsilon}}}d^{2}\ell_{f}^{2}\int_{z}^{d}\left(e^{\frac{2d}{\sqrt{\varepsilon}}}d^{2}\ell_{f}^{2}\right)^{n}\frac{\left(d-s\right)^{n}}{n!}\left\Vert w_{1}-w_{2}\right\Vert _{C\left(\left[0,d\right];L^{2}\left(\mathbb{R}^{2}\right)\right)}^{2}ds\\
 & \le & \left(e^{\frac{2d}{\sqrt{\varepsilon}}}d^{2}\ell_{f}^{2}\right)^{n+1}\frac{\left(d-z\right)^{n+1}}{\left(n+1\right)!}\left\Vert w_{1}-w_{2}\right\Vert _{C\left(\left[0,d\right];L^{2}\left(\mathbb{R}^{2}\right)\right)}^{2}.
\end{eqnarray*}
By the induction principle, we obtain \eqref{eq:3.26}. 
Furthermore,
 for every $\varepsilon>0$, it holds that
\[
\lim_{m\to\infty}\sqrt{\left(e^{\frac{2d}{\sqrt{\varepsilon}}}d^{2}\ell_{f}^{2}\right)^{m}\frac{\left(d-z\right)^{m}}{m!}}=0,
\]
this implies the existence of $m_{0}\in\mathbb{N}$ such
that
\[
\sqrt{\left(e^{\frac{2d}{\sqrt{\varepsilon}}}d^{2}\ell_{f}^{2}\right)^{m_{0}}\frac{\left(d-z\right)^{m_{0}}}{m_{0}!}}<1,
\]
which yields that  $\mathcal{G}^{m_{0}}$ is definitely a contraction mapping
on $C\left(\left[0,d\right];L^{2}\left(\mathbb{R}^{2}\right)\right)$.
It evidently provides us with the existence and uniqueness of solution
to the equation $\mathcal{G}^{n_{0}}\left(w\right)=w$ over the functional
space $C\left(\left[0,d\right];L^{2}\left(\mathbb{R}^{2}\right)\right)$
by the Banach fixed-point theorem. In addition, one has $\mathcal{G}^{n_{0}}\left(\mathcal{G}\left(w\right)\right)=\mathcal{G}\left(w\right)$
from the fact that $\mathcal{G}\left(\mathcal{G}^{n_{0}}\right)\left(w\right)=\mathcal{G}\left(w\right)$.
Combining this with the uniqueness of the fixed point of $\mathcal{G}^{n_{0}}$,
the equation $\mathcal{G}\left(w\right)=w$ admits a unique solution
in $C\left(\left[0,d\right];L^{2}\left(\mathbb{R}^{2}\right)\right)$.
Hence, we complete the proof of the lemma.\end{proof}
\begin{thm}\label{thm:17}
Assume that $u\in C\left(\left[0,d\right];L^{2}\left(\mathbb{R}^{2}\right)\right)$
is a unique solution of \eqref{eq:nonlinear1}, given by \eqref{eq:unon}
and satisfies the a priori condition
\[
\int_{\mathbb{R}^{2}\backslash\Theta_{\varepsilon}}e^{2\left|\rho\right|z}\left|\hat{u}\left(\rho,z\right)\right|^{2}d\rho\le Q\quad\mbox{ for all }z\in\left[0,d\right],
\]
where $Q$ is some positive constant. Let
$u_{\delta}^{\varepsilon}$ be the regularized solution given by \eqref{eq:uregunon}.
By choosing $\varepsilon=\dfrac{d^{2}}{\ln^{2}\left(\frac{1}{\delta}\right)}$,
we have the following estimate
\[
\left\Vert u\left(\cdot,z\right)-u_{\delta}^{\varepsilon}\left(\cdot,z\right)\right\Vert _{L^{2}\left(\mathbb{R}^{2}\right)}\le\left(\sqrt{Q}e^{\frac{1}{2}d^{2}\ell_{f}^{2}\left(d-z\right)}+\sqrt{3\left(d^{2}+1\right)}e^{\frac{3}{2}d^{2}\ell_{f}^{2}\left(d-z\right)}\right)\delta^{\frac{z}{d}}.
\]
\end{thm}
\begin{proof}
In the same vein to Theorem \ref{thm:11}, we shall estimate $\left\Vert u\left(\cdot,z\right)-u^{\varepsilon}\left(\cdot,z\right)\right\Vert _{L^{2}\left(\mathbb{R}^{2}\right)}$
and $\left\Vert u^{\varepsilon}\left(\cdot,z\right)-u_{\delta}^{\varepsilon}\left(\cdot,z\right)\right\Vert _{L^{2}\left(\mathbb{R}^{2}\right)}$
by means of the Fourier transform. We recall that $u^{\varepsilon}$ is
the regularized solution for exact data, which is given by
\begin{eqnarray}
u^{\varepsilon}\left(\xi,z\right) & = & \int_{\rho\in\Theta_{\varepsilon}}\left(\hat{g}\left(\rho\right)\cosh\left(\left(d-z\right)\sqrt{\lambda_{\rho,0}}\right)+\frac{\hat{h}\left(\rho\right)\sinh\left(\left(d-z\right)\sqrt{\lambda_{\rho,0}}\right)}{\sqrt{\lambda_{\rho,0}}}\right.\nonumber\\
 &  & \left.+\int_{z}^{d}\frac{\sinh\left(\left(s-z\right)\sqrt{\lambda_{\rho,0}}\right)}{\sqrt{\lambda_{\rho,0}}}\hat{F}\left(\rho,s,u^{\varepsilon} \right)ds\right)e^{2\pi i\left\langle \xi,\rho\right\rangle }d\rho.\label{eq:3.27}
\end{eqnarray}
Notice that the following inequalities hold for $\rho\in\Theta_{\varepsilon}$ and $0\le z\le s\le d$
\begin{equation}
\cosh^{2}\left(\left(d-z\right)\sqrt{\lambda_{\rho,0}}\right)\le e^{2\left(d-z\right)\left|\rho\right|}\le e^{\frac{2\left(d-z\right)}{\sqrt{\varepsilon}}},\label{eq:cosh}
\end{equation}
\begin{equation}
\frac{\sinh^{2}\left(\left(s-z\right)\sqrt{\lambda_{\rho,0}}\right)}{\lambda_{\rho,0}}\le\left(s-z\right)^{2}e^{2\left(s-z\right)\left|\rho\right|}\le d^{2}e^{\frac{2\left(s-z\right)}{\sqrt{\varepsilon}}},\label{eq:sinh}
\end{equation}
Observing \eqref{uhatnonlinear},\eqref{eq:unon} and \eqref{eq:3.27}, by Parseval's identity we have
\begin{eqnarray*}
\left\Vert u\left(\cdot,z\right)-u^{\varepsilon}\left(\cdot,z\right)\right\Vert _{L^{2}\left(\mathbb{R}^{2}\right)}^{2} & = & \int_{\rho\in\Theta_{\varepsilon}}\left|\hat{u}\left(\rho,z\right)-\hat{u}^{\varepsilon}\left(\rho,z\right)\right|^{2}d\rho+\int_{\rho\in\mathbb{R}^{2}\backslash\Theta_{\varepsilon}}\left|\hat{u}\left(\rho,z\right)\right|^{2}d\rho\\
 & \le & \int_{\rho\in\Theta_{\varepsilon}}\left|\int_{z}^{d}\frac{\sinh\left(\left(s-z\right)\sqrt{\lambda_{\rho,0}}\right)}{\sqrt{\lambda_{\rho,0}}}\left[\hat{F}\left(\rho,s,u\right)-\hat{F}\left(\rho,s,u^{\varepsilon}\right)\right]ds\right|^{2}d\rho\\
 &  & +e^{\frac{-2z}{\sqrt{\varepsilon}}}\int_{\rho\in\mathbb{R}^{2}\backslash\Theta_{\varepsilon}}e^{2z\left|\rho\right|}\left|\hat{u}\left(\rho,z\right)\right|^{2}d\rho.
\end{eqnarray*}
Thus, we obtain
\begin{equation}
\left\Vert u\left(\cdot,z\right)-u^{\varepsilon}\left(\cdot,z\right)\right\Vert _{L^{2}\left(\mathbb{R}^{2}\right)}^{2}\le e^{\frac{-z}{\sqrt{\varepsilon}}}Q+d^{2}\ell_{f}^{2}\int_{z}^{d}e^{\frac{2\left(s-z\right)}{\sqrt{\varepsilon}}}\left\Vert u\left(\cdot,s\right)-u^{\varepsilon}\left(\cdot,s\right)\right\Vert _{L^{2}\left(\mathbb{R}^{2}\right)}^{2}ds,\label{eq:3.30-1}
\end{equation}
where we have used \eqref{eq:sinh} and the Lipschitz property of
$F$.

Multiplying both sides of \eqref{eq:3.30-1} by $e^{\frac{2z}{\sqrt{\varepsilon}}}$
and putting $w_{1}\left(z\right)=e^{\frac{2z}{\sqrt{\varepsilon}}}\left\Vert u\left(\cdot,z\right)-u^{\varepsilon}\left(\cdot,z\right)\right\Vert _{L^{2}\left(\mathbb{R}^{2}\right)}^{2}$,
we get
\[
w_{1}\left(z\right)\le Q+d^{2}\ell_{f}^{2}\int_{z}^{d}w_{1}\left(s\right)ds.
\]
By using Gronwall's inequality, we have
\[
w_{1}\left(z\right)\le Qe^{d^{2}\ell_{f}^{2}\left(d-z\right)}.
\]
This is equivalent to 
\begin{equation}
\left\Vert u\left(\cdot,z\right)-u^{\varepsilon}\left(\cdot,z\right)\right\Vert _{L^{2}\left(\mathbb{R}^{2}\right)}\le\sqrt{Q}e^{\frac{1}{2}d^{2}\ell_{f}^{2}\left(d-z\right)}e^{\frac{-z}{\sqrt{\varepsilon}}}.\label{eq:3.31-1}
\end{equation}

In the same manner, taking into account \eqref{eq:3.27} and \eqref{eq:uregunon},
we can deduce from \eqref{eq:cosh} and \eqref{eq:sinh} that

\begin{eqnarray}
\left\Vert u^{\varepsilon}\left(\cdot,z\right)-u_{\delta}^{\varepsilon}\left(\cdot,z\right)\right\Vert _{L^{2}\left(\mathbb{R}^{2}\right)}^{2} & \le & 3e^{\frac{2\left(d-z\right)}{\sqrt{\varepsilon}}}\int_{\mathbb{R}^{2}}\left|\hat{g}\left(\rho\right)-\hat{g}_{\delta}\left(\rho\right)\right|^{2}d\rho+3d^{2}e^{\frac{2\left(d-z\right)}{\sqrt{\varepsilon}}}\int_{\mathbb{R}^{2}}\left|\hat{h}\left(\rho\right)-\hat{h}_{\delta}\left(\rho\right)\right|^{2}d\rho\nonumber \\
 &  & +3d^{2}\left\Vert \int_{z}^{d}e^{\frac{s-z}{\sqrt{\varepsilon}}}\left|\hat{F}\left(\cdot,s,u^{\varepsilon}\right)-\hat{F}\left(\cdot,s,u_{\delta}^{\varepsilon}\right)\right|ds\right\Vert _{L^{2}\left(\mathbb{R}^{2}\right)}^{2}\nonumber \\
 & \le & 3\left(d^{2}+1\right)\delta^{2}e^{\frac{2\left(d-z\right)}{\sqrt{\varepsilon}}}+3d^{2}\ell_{f}^{2}\int_{z}^{d}e^{\frac{2\left(s-z\right)}{\sqrt{\varepsilon}}}\left\Vert u^{\varepsilon}\left(\cdot,s\right)-u_{\delta}^{\varepsilon}\left(\cdot,s\right)\right\Vert _{L^{2}\left(\mathbb{R}^{2}\right)}^{2}ds,\label{eq:3.30}
\end{eqnarray}
where we have applied the elementary inequality $\left(a+b+c\right)^{2}\le3\left(a^{2}+b^{2}+c^{2}\right)$
for $a,b,c\ge0$,  the Lipschitz property of $F$ and the assumption
\textbf{(A1)}.

Now, multiplying both sides of \eqref{eq:3.30} by $e^{\frac{2z}{\sqrt{\varepsilon}}}$
and putting $w_{2}\left(z\right)=e^{\frac{2z}{\sqrt{\varepsilon}}}\left\Vert u^{\varepsilon}\left(\cdot,z\right)-u_{\delta}^{\varepsilon}\left(\cdot,z\right)\right\Vert _{L^{2}\left(\mathbb{R}^{2}\right)}^{2}$,
we obtain the following:
\begin{equation}
w_{2}\left(z\right)\le3\left(d^{2}+1\right)\delta^{2}e^{\frac{2d}{\sqrt{\varepsilon}}}+3d^{2}\ell_{f}^{2}\int_{z}^{d}w_{2}\left(s\right)ds.\label{eq:3.31}
\end{equation}
Thanks to Gronwall's inequality, we can deduce from \eqref{eq:3.31} that
\[
w_{2}\left(z\right)\le3\left(d^{2}+1\right)\delta^{2}e^{\frac{2d}{\sqrt{\varepsilon}}+3d^{2}\ell_{f}^{2}\left(d-z\right)},
\]
Therefore, we conclude that
\begin{equation}
\left\Vert u^{\varepsilon}\left(\cdot,z\right)-u_{\delta}^{\varepsilon}\left(\cdot,z\right)\right\Vert _{L^{2}\left(\mathbb{R}^{2}\right)}\le\sqrt{3\left(d^{2}+1\right)}\delta e^{\frac{1}{2}\left(\frac{2}{\sqrt{\varepsilon}}+3d^{2}\ell_{f}^{2}\right)\left(d-z\right)}.\label{eq:esti2}
\end{equation}

As a side note, this result also guarantees the stability of our regularized solution. Indeed, combining \eqref{eq:esti2} and \eqref{eq:3.31-1} together with $\varepsilon=\dfrac{d^{2}}{\ln^{2}\left(\frac{1}{\delta}\right)}$,
we obtain
\begin{eqnarray*}
\left\Vert u\left(\cdot,z\right)-u_{\delta}^{\varepsilon}\left(\cdot,z\right)\right\Vert _{L^{2}\left(\mathbb{R}^{2}\right)} & \le & \left\Vert u\left(\cdot,z\right)-u^{\varepsilon}\left(\cdot,z\right)\right\Vert _{L^{2}\left(\mathbb{R}^{2}\right)}+\left\Vert u^{\varepsilon}\left(\cdot,z\right)-u_{\delta}^{\varepsilon}\left(\cdot,z\right)\right\Vert _{L^{2}\left(\mathbb{R}^{2}\right)}\\
 & \le & \sqrt{Q}e^{\frac{1}{2}d^{2}\ell_{f}^{2}\left(d-z\right)}e^{\frac{-z}{\sqrt{\varepsilon}}}+\sqrt{3\left(d^{2}+1\right)}\delta e^{\frac{1}{2}\left(\frac{2}{\sqrt{\varepsilon}}+3d^{2}\ell_{f}^{2}\right)\left(d-z\right)}\\
 & \le & \left(\sqrt{Q}e^{\frac{1}{2}d^{2}\ell_{f}^{2}\left(d-z\right)}+\sqrt{3\left(d^{2}+1\right)}e^{\frac{3}{2}d^{2}\ell_{f}^{2}\left(d-z\right)}\right)\delta^{\frac{z}{d}},
\end{eqnarray*}
which completes the proof of the theorem.
\end{proof}

\section{Numerical examples}\label{Sec:numerical}
\subsection{Example 1}
We
consider the problem \eqref{eq:helmholtz} in a mock-up, a nearly rectangular box, whose size is approximately 
1.0 m $\times$ 1.0 m $\times$ 0.5 m. It is worth noting that the speed of light 
 in free space is about $3\times10^{8}\;\mbox{m}\cdot s^{-1}$, 
and the wavelength $\lambda$ of a 100 MHz electromagnetic (radio)
wave is nearly 3.0 m. Hereby, this leads to the wave number $k=1/3$ $\mbox{m}^{-1}$.
Suppose that the forcing term $f$ almost agrees to zero outside the box whilst in the box we have
\[
f\left(x,y,z\right)=-xy\left(z-\frac{1}{2}\right)^{2}\left(12+\frac{1}{9}\left(z-\frac{1}{2}\right)^{2}\right),
\]
and $g,h\equiv0$ in the whole field.

The solution is explicitly known by
$xy\left(z-\dfrac{1}{2}\right)^{4}$. To keep track of what we have done in Section \ref{Sec:truncate}, the presence of $h_{\delta}$ is not considered, and we provide herein the measured functions $f_{\delta}$ and $g_{\delta}$ inside the mock-up by
\[
f_{\delta}\left(x,y,z\right)=f\left(x,y,z\right)\left(1+\frac{\delta}{0.3167506677}\right),\quad g_{\delta}\left(x,y\right)=\delta.
\]

Let us denote by $\left(0,c\right)^{2}=\left(0,c\right)\times\left(0,c\right)$
where $c>0$. We aim at considering solutions in
terms of the Fourier transform in the case $A_{1}$ since this is the unstable case. From \eqref{eq:uhatdelta},
we have the following representation:
\begin{equation}
\hat{u}_{\delta}^{\varepsilon}\left(\rho,z\right)=\hat{g}_{\delta}^{\varepsilon}\left(\rho\right)\cosh\left(\left(\frac{1}{2}-z\right)\sqrt{\lambda_{\rho,\frac{1}{3}}}\right)-\frac{1}{\sqrt{\lambda_{\rho,\frac{1}{3}}}}\int_{z}^{\frac{1}{2}}\hat{f}_{\delta}^{\varepsilon}\left(\rho,s\right)\sinh\left(\left(s-z\right)\sqrt{\lambda_{\rho,\frac{1}{3}}}\right)ds,\label{eq:4.3}
\end{equation}
where we have expressed the Fourier coefficients by
\begin{equation}
\hat{g}_{\delta}^{\varepsilon}\left(\rho\right)=\delta\int_{\left(0,1\right)^{2}}e^{-2\pi i\left\langle \rho,\xi\right\rangle }d\xi,\quad\hat{f}_{\delta}^{\varepsilon}\left(\rho,s\right)=\int_{\left(0,1\right)^{2}}f\left(\xi,s\right)\left(1+\frac{\delta}{0.3167506677}\right)e^{-2\pi i\left\langle \rho,\xi\right\rangle }d\xi,\quad \rho\in\tilde{\Theta}_{\varepsilon}:=\Theta_{\varepsilon}\cap A_{1}.\label{eq:4.4}
\end{equation}

Simultaneously, we show the representation of the exact solution
\begin{equation}
\hat{u}\left(\rho,z\right)=\left(z-\dfrac{1}{2}\right)^{4}\int_{\left(0,1\right)^{2}}\xi_{1}\xi_{2}e^{-2\pi i\left\langle \rho,\xi\right\rangle }d\xi,\quad\rho\in A_{1}.\label{eq:4.5}
\end{equation}

Here, we do not pay more attention
to the integrals in \eqref{eq:4.4}, but the integral in \eqref{eq:4.3}
that runs from $z$ to $1/2$ with respect to $s$ should
be solved numerically. From the numerical point of view, we apply the Gauss-Legendre
quadrature method by using Legendre-Gauss nodes and weights on the
interval $\left[z_{0},1/2\right]$ with truncation order
$N\in\mathbb{N}$. Since $\left|\tilde{\Theta}_{\varepsilon}\right|\to\infty$ as $\varepsilon\to0^{+}$,
a uniform grid corresponding to $\rho\in\tilde{\Theta}_{\varepsilon}$
is arbitrarily generated by the partition $1/\left(30\sqrt{\varepsilon}\right)$. 

In a common manner, the whole process is established
as follows:
\begin{enumerate}
\item Define the \emph{a priori} constant $M_{0}$, then the regularization
parameter is explicitly computed by
\[
\varepsilon\left(\delta\right)=\left(\frac{1}{9}+4\ln^{2}\left(48\delta\right)\right)^{-1};
\]

\item Determine the set $\tilde{\Theta}_{\varepsilon}$, consider a fixed
$z_{0}$ for implementation and choose the truncation order $N$;
\item Generate a vector including $\rho=\left(\rho_{1},\rho_{2}\right)$
in $\tilde{\Theta}_{\varepsilon}$. Then, compute $\hat{u}_{\varepsilon}^{\delta}\left(\rho,z_{0}\right)$
and $\hat{u}\left(\rho,z_{0}\right)$ on this uniform grid;
\item Compute the $\ell_2$ error:
\[
E\left(z_{0}\right)=\sqrt{\frac{1}{\mbox{card}\left(\tilde{\Theta}_{\varepsilon}\right)}\sum_{\rho\in\tilde{\Theta}_{\varepsilon}}\left|\hat{u}\left(\rho,z_{0}\right)-\hat{u}_{\varepsilon}^{\delta}\left(\rho,z_{0}\right)\right|^{2}};
\]

\item Draw 3D graphs for the modulus of the solutions.
\end{enumerate}

\begin{figure}
\centering
\subfloat[Exact]{\includegraphics[scale=0.35]{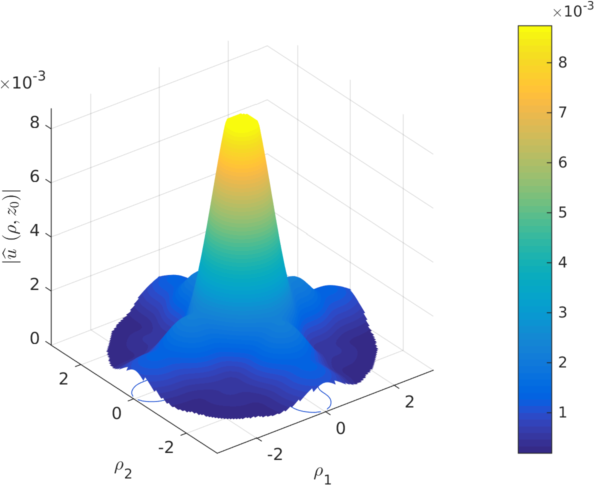}

}\subfloat[Regularized]{\includegraphics[scale=0.35]{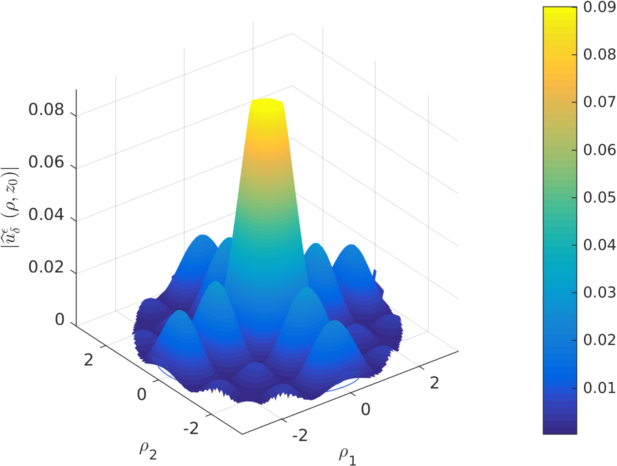}

}

\protect\caption{Modulus of the exact and regularized solutions with noise amplitude $\delta=10^{-1}$.\label{fig:1}}

\end{figure}

\begin{figure}
\centering
\subfloat[Exact]{\includegraphics[scale=0.35]{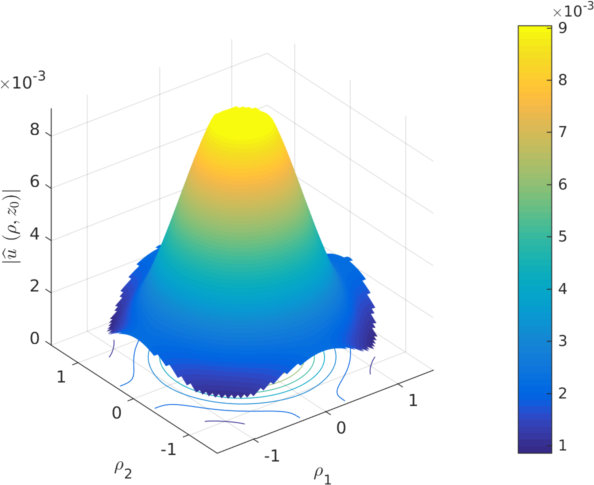}

}\subfloat[Regularized]{\includegraphics[scale=0.35]{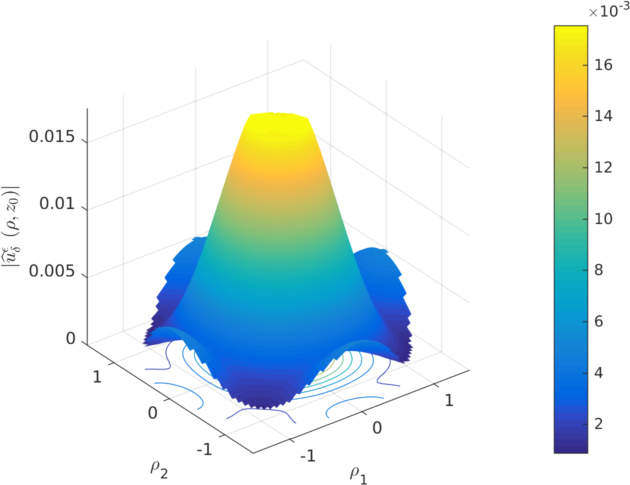}

}

\protect\caption{Modulus of the exact and regularized solutions with noise amplitude
 $\delta=10^{-2}$.\label{fig:2}}
\end{figure}

\begin{figure}
\centering
\subfloat[Exact]{\includegraphics[scale=0.35]{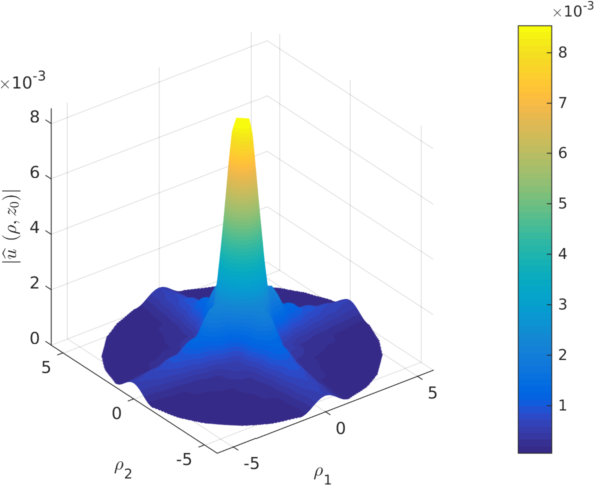}

}\subfloat[Regularized]{\includegraphics[scale=0.35]{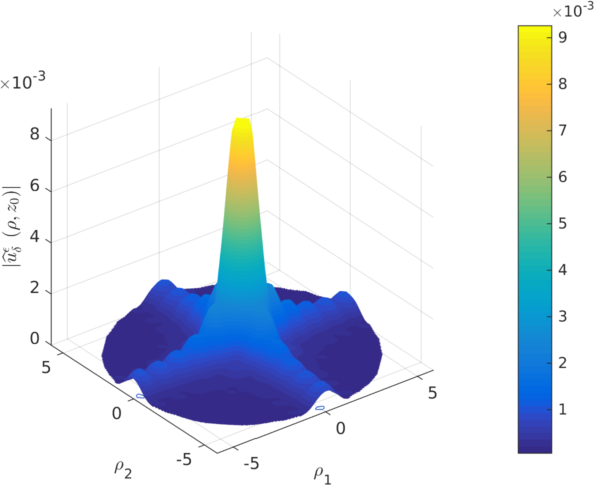}

}

\protect\caption{Modulus of the exact and regularized solutions with noise amplitude
 $\delta=10^{-3}$.\label{fig:3}}
\end{figure}

\begin{figure}
\centering
\subfloat[Exact]{\includegraphics[scale=0.35]{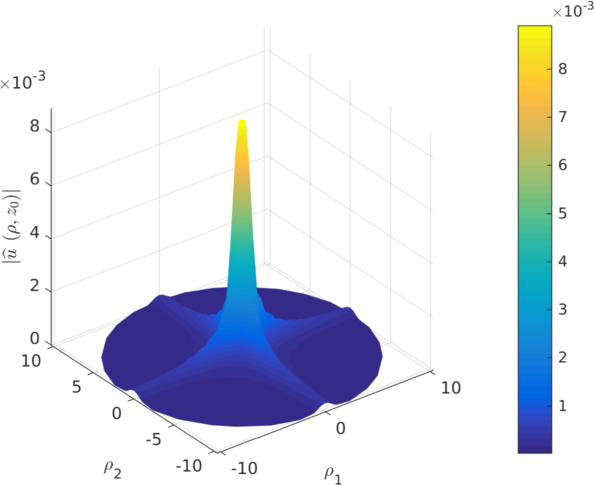}

}\subfloat[Regularized]{\includegraphics[scale=0.35]{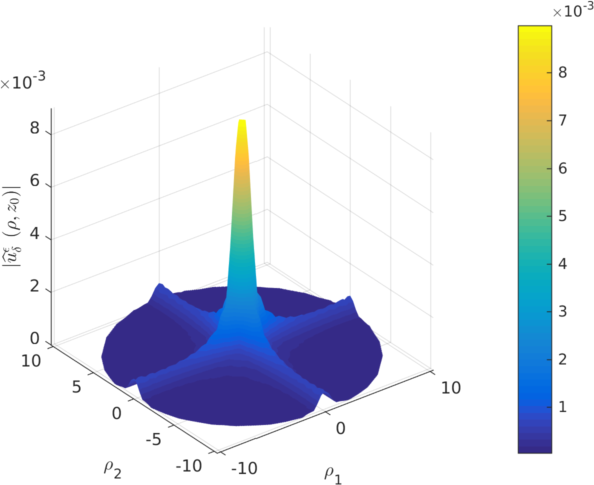}

}

\protect\caption{Modulus of the exact and regularized solutions with noise amplitude
 $\delta=10^{-4}$.\label{fig:4}}
\end{figure}


In what follows, we consider numerical solutions of \eqref{eq:4.3}
and \eqref{eq:4.5} in a 3D computational dimensionless domain of
$\tilde{\Theta}_{\varepsilon}\times\left(0,1/2\right]$, i.e.
$d=1/2$. In computational process, we choose $N=5$ and
let $M_{0}=1/48$ due to the energy $\left\Vert u\left(\cdot,0\right)\right\Vert _{L^{2}\left(\mathbb{R}^{2}\right)}$
from the exact solution. Table \ref{tab:1}
gives the error between the exact and regularized solutions
by means of the Fourier transform at each point $z_{0}$. Besides, the approximation of the regularized solution can be seen in Figure \ref{fig:1}-\ref{fig:4} where $z_{0}=0.05$ is considered.
Additionally, the color bars in these figures show the accuracy of the approximation when $\delta$
is smaller and smaller. Moreover, one may see that the region
becomes greater in width, then it explains clearly how the reconstruction works.

We also go further by presenting the method in \cite{Xiong10} where
the author obtained the same convergence rate. Due to the
fact that we are implementing the case $\rho\in\tilde{\Theta}_{\varepsilon}$,
the solution modified by that method can be viewed as a successful
combination of the truncation and quasi-boundary methods and thus
it is given by the following:
\begin{equation}
\hat{u}_{\delta}^{\varepsilon}\left(\rho,z\right)=\hat{g}_{\delta}^{\varepsilon}\left(\rho\right)\frac{\cosh\left(\left(\frac{1}{2}-z\right)\sqrt{\lambda_{\rho,\frac{1}{3}}}\right)}{1+\alpha\cosh\left(\frac{1}{2}\sqrt{\lambda_{\rho,\frac{1}{3}}}\right)}-\frac{1}{\sqrt{\lambda_{\rho,\frac{1}{3}}}}\int_{z}^{\frac{1}{2}}\hat{f}_{\delta}^{\varepsilon}\left(\rho,s\right)\frac{\sinh\left(\left(s-z\right)\sqrt{\lambda_{\rho,\frac{1}{3}}}\right)-\alpha\sinh\left(\left(1-s+z\right)\sqrt{\lambda_{\rho,\frac{1}{3}}}\right)}{1+\alpha\sinh\left(\frac{1}{2}\sqrt{\lambda_{\rho,\frac{1}{3}}}\right)}ds,\label{eq:xiong}
\end{equation}
where $\alpha=\frac{\delta}{M_{0}}$ represents the regularization parameter.

As shown in Table \ref{tab:2}, the numerical results mostly keep unchanged (compared to Table \ref{tab:1} in the order of error). 
They also present a great convergence
nearby the final data with the same order. In addition, it is slightly
better than the corresponding results in Table \ref{tab:1} because of
the region $\tilde{\Theta}_{\varepsilon}$ whereas the results nearby
the original $z=0$ somewhat get slower. Thus, saying
that our speed of convergence is similar to their rate is possibly correct. 

\noindent \begin{center}
\begin{table}
\noindent \begin{centering}
\begin{tabular}{|c|c|c|c|c|}
\hline 
$\delta$ & $E\left(0.4\right)$ & $E\left(0.25\right)$ & $E\left(0.1\right)$ & $E_{2}\left(0.05\right)$\tabularnewline
\hline 
1.0E-01 & 1.40421792E-02 & 1.44104551E-02 & 1.53538523E-02 & 1.58767301E-02\tabularnewline
\hline 
1.0E-02 & 2.91444011E-03 & 2.94877381E-03 & 3.05224512E-03 & 3.11545162E-03\tabularnewline
\hline 
1.0E-03 & 7.61081237E-05 & 8.11322809E-05 & 1.06859508E-04 & 1.32533217E-04\tabularnewline
\hline 
1.0E-04 & 4.95852862E-06 & 6.51466288E-06 & 3.95372158E-05 & 7.97083661E-05\tabularnewline
\hline 
\end{tabular}
\par\end{centering}

\protect\caption{Numerical results in Example 1.\label{tab:1}}
\end{table}

\par\end{center}

\noindent \begin{center}
\begin{table}
\noindent \begin{centering}
\begin{tabular}{|c|c|c|c|c|}
\hline 
$\delta$ & $E\left(0.4\right)$ & $E\left(0.25\right)$ & $E\left(0.1\right)$ & $E_{2}\left(0.05\right)$\tabularnewline
\hline 
1.0E-01 & 2.10678921E-03 & 2.67347266E-03 & 1.36686573E-02 & 2.00115229E-02\tabularnewline
\hline 
1.0E-02 & 1.83489554E-03 & 1.67705992E-03 & 8.36907677E-03 & 1.23680304E-02\tabularnewline
\hline 
1.0E-03 & 6.78544187E-05 & 6.54322866E-05 & 3.03290363E-04 & 5.56245094E-04\tabularnewline
\hline 
1.0E-04 & 4.64459553E-06 & 4.49543976E-06 & 3.42596062E-05 & 6.95604817E-05\tabularnewline
\hline 
\end{tabular}
\par\end{centering}

\protect\caption{Numerical results from the regularized solution \eqref{eq:xiong} by Xiong et al. \cite{Xiong10}. \label{tab:2}}
\end{table}

\par\end{center}

\subsection{Example 2}
The numerical results presented in Table \ref{tab:3} are to implement the pseudo-nonlinear case where we do not mind the condition \textbf{(A3)} as well as $M_0$ or the energy $\left\Vert u\left(\cdot,0\right)\right\Vert _{L^{2}\left(\mathbb{R}^{2}\right)}$. Consequently, we choose herein $d=\pi/\sqrt{3}$ and $k=\sqrt{5}$ which imply $k\ge\frac{\pi}{2d}$. While we artificially stimulate a small quantity $\delta$, albeit without noise along with the data in this example, we gain the parameter $\varepsilon$ by the relation $\varepsilon=\dfrac{d^{2}}{\ln^{2}\left(\frac{1}{\delta}\right)}$. The exact solution and forcing term $F$ in this case are, respectively, tested by $e^{-\frac{1}{2}(x^2+y^2)}\cos\left(\frac{z\pi}{d}\right)$ and
\[
F(u)=-k^2u + (x^2 + y^2)e^{-\frac{1}{2}(x^2+y^2)}\cos\left(\frac{z\pi}{d}\right).
\]

After some arrangements, our regularized solution in the Fourier transform can be computed by
\begin{eqnarray}
\hat{u}^{\varepsilon}(\rho,z) 
& = &
-2\pi e^{-2\pi^2\left(\rho_{1}^2+\rho_{2}^2\right)}
\left[\cosh\left((d-z)\sqrt{\lambda_{\rho,0}}\right)\right.
\nonumber
\\
 & & \left. +
\frac{d\left(2-4\pi^2(\rho_{1}^2+\rho_{2}^2)\right)}{\pi^2 + \lambda_{\rho,0}d^2}
\left(\frac{\pi}{\sqrt{\lambda_{\rho,0}}}\sin\left(\frac{z\pi}{d}\right)\sinh((d-z)\sqrt{\lambda_{\rho,0}})
-d\left(1+\cos\left(\frac{z\pi}{d}\right)\cosh((d-z)\sqrt{\lambda_{\rho,0}})\right)\right)
\right]
\nonumber\\
& & -k^2\int_{z}^{d}\frac{\sinh\left((s-z)\sqrt{\lambda_{\rho,0}}\right)}{\sqrt{\lambda_{\rho,0}}}
\hat{u}^{\varepsilon}(\rho,s)ds\quad\text{for}\;\rho\in\Theta_{\varepsilon}.\label{eq:abc}
\end{eqnarray}

Despite dealing with the exact data, the problem is still difficult because of the integral equation \eqref{eq:abc}. There are many approaches to approximate this Volterra integral equation of second kind with respect to $z$ for each $\rho$. However,  since we are just in the part of numerical tests, it can be solved  by computing an approximate solution under construction of an inverse iterative scheme. The scheme can be derived from the direct scheme proposed in \cite{TTKT15}. In fact, we first denote $r(\rho,z)$ by the first and second quantities in (\ref{eq:abc}). If the starting point, denoted by $u^{\varepsilon}_{[M]}(\rho)$, is defined by $\hat{g}(\rho)=-2\pi e^{-2\pi^2\left(\rho_{1}^2+\rho_{2}^2\right)}$, we compute $u^{\varepsilon}_{[i]}(\rho)\approx u^{\varepsilon}(\rho,z_{i}),i=\overline{0,M-1}$ where the interval $[0,d]$ is divided into $M$ parts, as follows:
\begin{eqnarray}
\hat{u}^{\varepsilon}_{[i]}(\rho) &=&
 r(\rho,z_{i}) -k^2\sum_{i\le j\le M-1}\int_{z_j}^{z_{j+1}}\frac{\sinh\left((s-z_{i})\sqrt{\lambda_{\rho,0}}\right)}{\sqrt{\lambda_{\rho,0}}}\hat{u}^{\varepsilon}_{[j+1]}(\rho)ds\nonumber\\
&=& r(\rho,z_{i}) -\frac{k^2}{\lambda_{\rho,0}}\sum_{i\le j\le M-1}\left[\cosh\left((z_{i}-z_{j+1})\sqrt{\lambda_{\rho,0}}\right)
- \cosh\left((z_{i}-z_{j})\sqrt{\lambda_{\rho,0}}\right)\right]\hat{u}^{\varepsilon}_{[j+1]}(\rho).\nonumber
\end{eqnarray}

\noindent \begin{center}
\begin{table}
\noindent \begin{centering}
\begin{tabular}{|c|c|c|c|c|}
\hline 
$\delta$ & $E\left(1.45\right)$ & $E\left(1.08\right)$ & $E\left(0.90\right)$ & $E_{2}\left(0.36\right)$\tabularnewline
\hline 
1.0E-03 & 5.13693230E-02 & 2.91263157E-01 & 1.97712392E-01 & 2.39488777E-01\tabularnewline
\hline 
1.0E-05 & 1.64503956E-02 & 1.17228155E-01 & 5.19008818E-02 & 3.95350268E-02\tabularnewline
\hline 
1.0E-07 & 1.05357366E-02 & 4.53842881E-02 & 2.61078016E-02 & 2.29135650E-02\tabularnewline
\hline 
1.0E-09 & 5.22770995E-03 & 1.07845343E-02 & 1.37167983E-02 & 1.94734413E-02\tabularnewline
\hline 
\end{tabular}
\par\end{centering}

\protect\caption{Numerical results in Example 2.\label{tab:3}}
\end{table}

\par\end{center}

For testing, we use $M=50$ for all numerical results in this example. As shown in Table \ref{tab:3}, for smaller values of $\delta$ the error $E$ reduces but not drastically. It can be explained theoretically that the analysis we obtained in Theorem \ref{thm:17} yields the error estimate including functions that have large values indeed. It thus hinders the speed of the convergence. Additionally, the parameter $\varepsilon$ in this example may spread slower than the one in Example 1 because of the larger number $d$. To sum up,
our numerical implementation confirms the theoretical expectations.

\section{Conclusion and discussion}\label{Sec:con}

We have proposed a truncation method for the regularization of the non-homogeneous
Helmholtz equation in three dimensions along with Cauchy data. We prove that the problem \eqref{eq:helmholtz} is ill-posed in general. As the relation of $k$ and $d$ in \textbf{(A3)} appears, we have pointed out, from the mathematical point of view, the example (in Lemma \ref{lem:6}) that naturally leads to the catastrophic growth. That is our first novelty. The second is that our analysis not only guarantees an extension
of the method from the homogeneous cases in \cite{RR06} (and so far others in \cite{FFC11,Xiong10}) to the non-homogeneous
case, but also gives a short extension to the nonlinear case. Interestingly, all linear and nonlinear cases presented in Section \ref{Sec:truncate} and \ref{Sec:extension} respectively, are of the same convergence rate $\mathcal{O}\left(\delta^{\frac{z}{d}}\right)$. It is worth noting that while the energy-like restriction of the solution is required in the main theorem (Theorem \ref{thm:11}) of Section \ref{Sec:truncate}, there is no such dependence in Theorem \ref{thm:13} and that leads to the occurrence of Section \ref{Sec:extension}. Simple numerical examples are provided to indicate that the proposed method works well. We are also aware of the possibilities of solving the modified Helmholtz equation \cite{NTN13} and the elliptic sine-Gordon equation with Cauchy data in the present paper.

For our forthcoming aims of study, the impetus of analyzing issues related to numerical simulations is necessary. A few examples are in order: the comparison of GLSFEM, QSFEM and RFFEM for the problem of minimizing the pollution effect in \cite{DBB99}, the UWVF method for the ultrasound problems in \cite{HMKWH05}, the comparison of several boundary element regularization methods for the Cauchy problem of Helmholtz equations in \cite{Mar04}, an improved numerical algorithm for nonlinear self-focusing model of time-harmonic electromagnetic waves in optical propagation in \cite{FT05}, the FEM for the time-domain scattering problem in 2D cavities in \cite{VW02}, etc.

\section*{Acknowledgment}
A significant part of this work was carried out while the first author
was visiting the School of Mathematics and Statistics at The University
of New South Wales in Sydney. Their hospitality is gratefully acknowledged.
The authors also desire to thank the handling editor and anonymous
referees for their helpful comments on this paper.

\bibliography{mybib}

\end{document}